\documentclass[10pt,reqno]{amsart}

\usepackage{amsmath}
\usepackage{amsfonts}
\usepackage{amssymb}
\usepackage{amsthm}
\usepackage{footnote}

\usepackage{graphicx}
\usepackage{subfigure}
\usepackage{amssymb}
\usepackage{epstopdf}
\usepackage{color}

\usepackage{ifthen} 

\provideboolean{shownotes} 
\setboolean{shownotes}{true} 
\newcommand{\margnote}[1]{
\ifthenelse{\boolean{shownotes}}%
{\marginpar{\raggedright\tiny\texttt{#1}}}%
{}%
}

\newcommand{\hole}[1]{
\ifthenelse{\boolean{shownotes}}%
{\begin{center} \fbox{ \rule {.25cm}{0cm}
\rule[-.1cm]{0cm}{.4cm} \parbox{.85\textwidth}{\begin{center}
\texttt{#1}\end{center}} \rule {.25cm}{0cm}}\end{center}}
{}
}

\numberwithin{equation}{section}

\theoremstyle{plain}

\newtheorem{lemma}{Lemma}[section]

\newtheorem{theorem}[lemma]{Theorem}
\newtheorem{proposition}[lemma]{Proposition}
\newtheorem{corollary}[lemma]{Corollary}

\theoremstyle{definition}

\newtheorem{remark}[lemma]{Remark}
\newtheorem{definition}[lemma]{Definition}

\theoremstyle{remark}

\newcommand{\Id}{\mathbb{I}}
\newcommand{\Q}{\mathbb{Q}}
\newcommand{\A}{\mathbb{A}}
\newcommand{\F}{\mathbb{F}}
\newcommand{\R}{\mathbb{R}}
\newcommand{\C}{\mathbb{C}}

\newcommand{\M}{\mathbb{M}}
\newcommand{\Hm}{\mathbb{H}}
\newcommand{\W}{\mathbb{W}}

\newcommand{\cA}{{\mathcal{A}}}

\renewcommand{\Re}{\mathrm{Re}\,} 
\renewcommand{\Im}{\mathrm{Im}\,}
\newcommand{\tr}{\mathrm{tr}\,}

\newcommand{\sigmaP}{\sigma\eqref{eq:p}}
\newcommand{\sigmaQ}{\sigma\eqref{eq:q}}

\newcommand{\pandq}{\eqref{eq:p} and \eqref{eq:q}}

\newcommand{\re}[1]{\textrm{Re}\left(#1\right)}
\newcommand{\im}[1]{\textrm{Im}\left(#1\right)}
\def\sgn{\mathop{\rm sgn}}
\title[Periodic traveling waves in the sine-Gordon equation]{On the stability analysis of periodic sine-Gordon traveling waves}

\author[C.K.R.T. Jones]{Christopher K. R. T. Jones}

\address[C.K.R.T. Jones]{ Department of Mathematics\\University of North Carolina\\Chapel Hill, NC 27599 (U.S.A.)}

\email{ckrtj@amath.unc.edu}

\author[R. Marangell]{Robert Marangell$^*$}
\thanks{*{\em Corresponding Author:} University of Sydney. Email: robert.marangell@usyd.eud.au. \\Tel: +61 2 9351 5763. Fax: +61 2 9351 4534.}

\address{{\rm (R. Marangell)} School of Mathematics and Statistics F07\\University of Sydney\\Sydney NSW 2006 (Australia)}

\email{robert.marangell@usyd.edu.au}

\author[P.D. Miller]{Peter D. Miller}

\address{{\rm (P. D. Miller)} Department of Mathematics\\University of Michigan\\Ann Arbor, MI 48109 (U.S.A.)}

\email{millerpd@umich.edu}

\author[R.G. Plaza]{Ram\'on G. Plaza}

\address{{\rm (R. G. Plaza)} Departamento de Matem\'aticas y Mec\'anica\\IIMAS-FENOMEC\\Universidad Nacional Aut\'onoma de M\'exico\\Apdo. Postal 20-726, C.P. 01000 M\'exico D.F. (Mexico)}

\email{plaza@mym.iimas.unam.mx}


\begin{document}



%
%
%
%
%
%
%
%
%
%
%

\begin{abstract}
We study the spectral stability properties of periodic traveling waves in the sine-Gordon equation, including waves of both subluminal and superluminal propagation velocities as well as waves of both librational and rotational types.  We prove that only subluminal rotational waves are spectrally stable and establish exponential instability in the other three cases.  Our proof corrects a frequently cited one given by Scott \cite{Sco1}.
\end{abstract}

\maketitle



\section{Introduction}

Consider the sine-Gordon equation \cite{Sco2} in laboratory coordinates,
\begin{equation}
\label{eq:sinegordon}
u_{tt} - u_{xx} + \sin u = 0,
\end{equation}
where $u$ is a scalar and $(x,t) \in \R \times [0,+\infty)$. Although it first appeared in the study of the geometry of surfaces with negative Gaussian curvature (see \cite{Eis0}), the sine-Gordon equation describes a great variety of physical phenomena as well, such as the propagation of magnetic flux on a Josephson line \cite{Sco2}, elementary particles \cite{PeSky}, modeling of fermions in the Thirring model \cite{Cole75}, the propagation of crystal dislocations \cite{FreKo}, and the oscillations of an array of rigid pendula rotating under gravity about a common axis with nearest-neighbor torque coupling \cite{Dra0}, among others. A comprehensive account of these and other physical applications can be found in the review article by Barone \textit{et al.} \cite{BEMS}.

Our paper is concerned with the linearized (spectral) stability of the family of periodic traveling wave solutions to the sine-Gordon equation \eqref{eq:sinegordon}.  
%
%
%
%
%
To describe general traveling wave solutions,
one first goes into a  frame of reference moving with constant velocity $c$, which in turn amounts to rewriting \eqref{eq:sinegordon} in the new independent variables $z=x-ct$ and $\tau=t$.  Thus with $u(x,t)=v(z,\tau)$,
\begin{equation}\label{eq:travelingsinegordon}
(c^2-1)v_{zz} - 2 c v_{z \tau} + v_{\tau\tau} + \sin(v) = 0. 
\end{equation}
In what follows, we will always assume that $c \neq \pm 1$. 
A traveling wave solution of the sine-Gordon equation \eqref{eq:sinegordon} is by definition a stationary ($\tau$-independent) solution of \eqref{eq:travelingsinegordon}.  Making the ansatz $v(z,\tau)=f(z)$ implies that $f$ solves the nonlinear pendulum equation, that is, the ordinary differential equation
\begin{equation}\label{eq:pendulum}
(c^2-1)f''(z) + \sin(f(z))=0, \quad ':=\frac{d}{dz}.
\end{equation}
The pendulum equation \eqref{eq:pendulum} can be integrated once to obtain: 
\begin{equation}\label{eq:pendint}
\frac12(c^2-1)f'(z)^2+1-\cos(f(z)) = E
\end{equation}
where $E$ is a constant of integration (the total energy).   
Clearly, these equations are invariant under the shift $f(z)\mapsto f(z)+2\pi$. 
\begin{definition}[Wave speed dichotomy of traveling waves] Traveling wave solutions $f(z)$ with wave speeds satisfying $c^2<1$ (respectively $c^2>1$) are called \emph{subluminal} (respectively \emph{superluminal}).  
\end{definition} 
Representative phase portraits corresponding to subluminal and superluminal traveling waves are illustrated in Figure~\ref{fig:pendulum}.
\begin{figure}[h]
\subfigure[Subluminal waves: $c^2 < 1$]{\includegraphics[scale=.4, trim=2cm 0cm 0cm 0cm, clip=true]{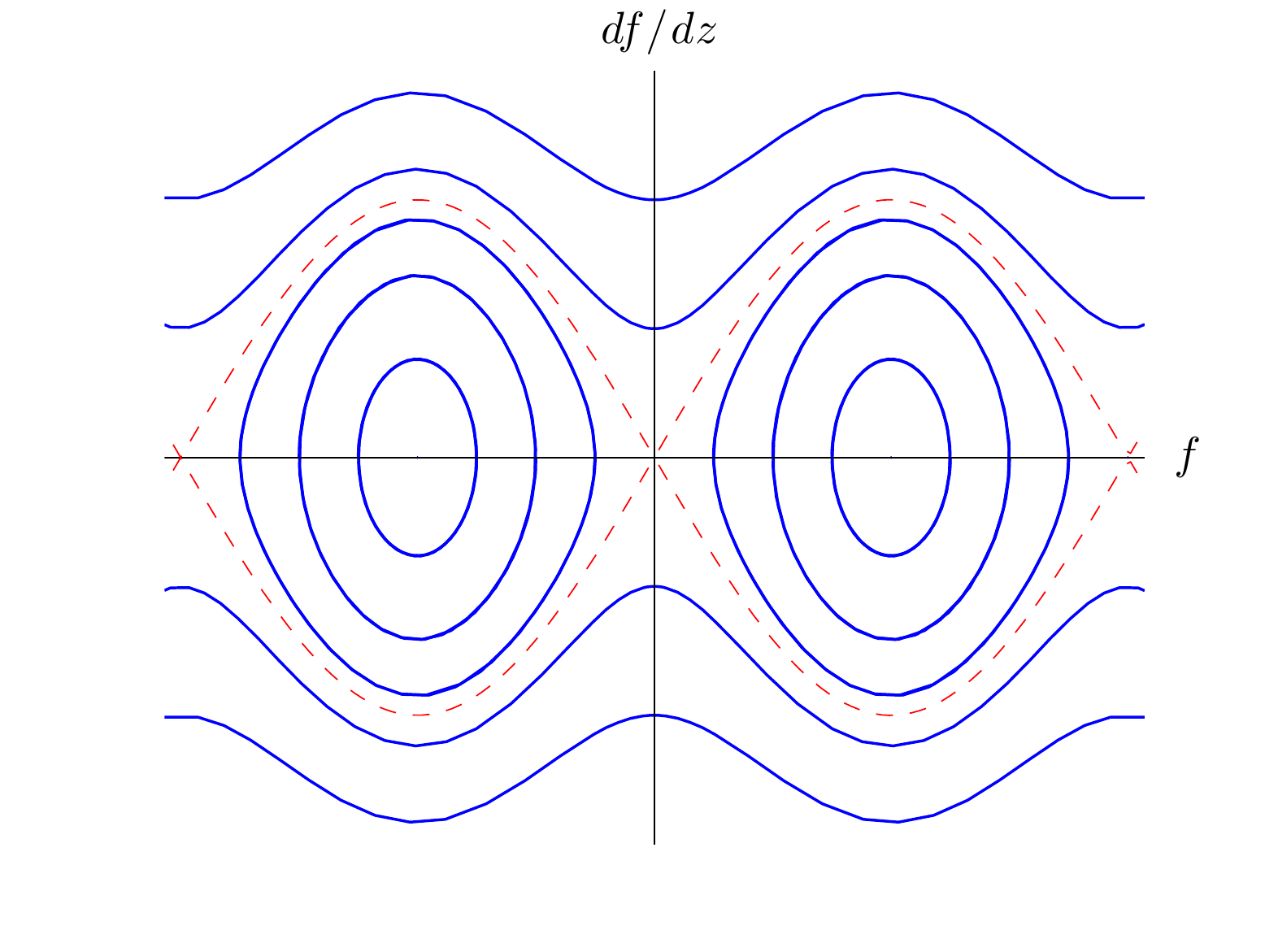}}
\subfigure[Superluminal waves: $c^2 > 1$]{\includegraphics[scale=.4, trim=2cm 0cm 0cm 0cm, clip=true]{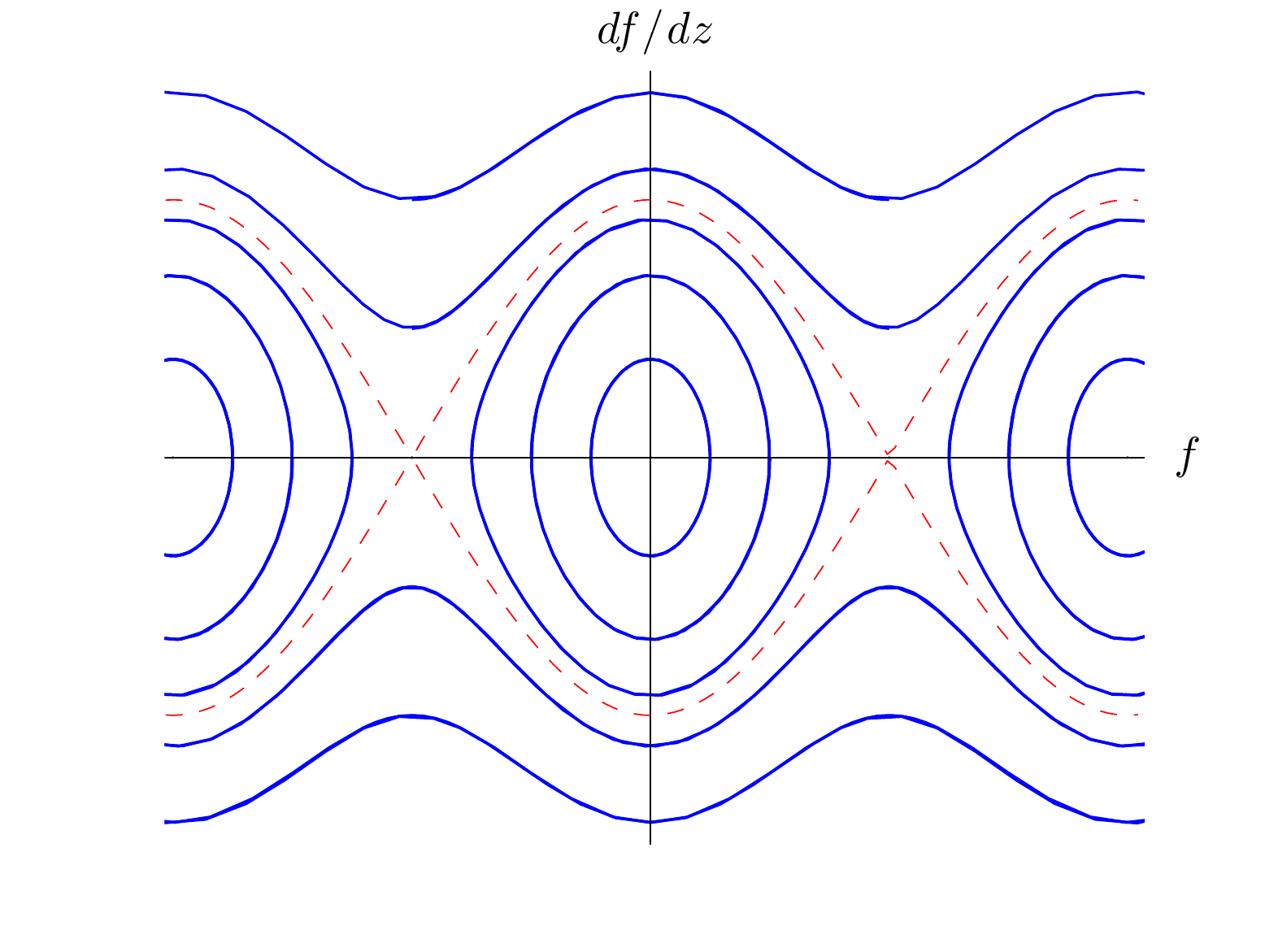}}
\caption{Phase portraits of the pendulum equation \eqref{eq:pendulum} showing both librational waves (closed orbits inside the separatrix) and rotational waves (orbits outside the separatrix). The separatrix is depicted by a dotted curve. }
\label{fig:pendulum}
\end{figure}
It is  easy to confirm  that except for solutions corresponding to the separatrices, all traveling wave solutions $f(z)$ are periodic modulo $2\pi$.  The second dichotomy for
such traveling waves is the following.
\begin{definition}[Energy dichotomy of periodic traveling waves] Solutions $f(z)$ to the pendulum equation \eqref{eq:pendulum} whose orbits in the phase plane lie outside the separatrix are called {\em rotational} waves. Solutions whose orbits in the phase plane are within the separatrix are called {\em librational} waves.
\end{definition}
It is easy to see that librational waves correspond to energies in the range $0<E<2$ and rotational waves correspond to energies with either $E<0$ (in the subluminal case) or $E>2$ (in the superluminal case).

%

\begin{remark}
In  classical mechanics  (cf.\ Goldstein \cite{Golds2ed}; see also \cite{BuM2}) the term \emph{libration} is borrowed from the astronomical literature describing periodic motions in which both the position and the momentum are periodic functions with same frequency. The term \emph{rotation} (sometimes also called  \emph{circulation} or \emph{revolution}) is used to characterize the kind of periodic motion in which the momentum is periodic but the position is no longer bounded. Increments by a period in the position, however, produce no essential change in the state of the system. 
In the sequel we will refer to both librational and rotational waves as ``periodic traveling waves,'' regardless of whether or not $f$ returns to the same value.
\end{remark}

To study the stability of the periodic traveling wave $f$, we make the substitution $v(z,\tau):=f(z)+w(z,\tau)$ in \eqref{eq:travelingsinegordon} and consider the implied behavior of the perturbation $w(z,\tau)$.  In the simplest approximation as the perturbation is initially small, we linearize about $w=0$ to obtain the following:
\begin{equation}\label{eq:linearpde}
(c^2-1)w_{zz} - 2 c w_{z\tau} + w_{\tau\tau} + \cos(f(z)) w = 0.
\end{equation} 
We seek separated solutions with exponential growth rate $\lambda\in\mathbb{C}$ of the form
\begin{equation}
w(z,\tau)=p(z)e^{\lambda\tau},
\end{equation}
which reduces \eqref{eq:linearpde} to a linear ordinary differential equation for $p$:
%
\begin{equation*}\label{eq:p}\tag{{\bf P}}
p'' - 2 c \gamma \lambda p' +\gamma\left[\lambda^2 + \cos(f(z))\right] p = 0, \quad ' := \frac{d}{dz}
\end{equation*}
where $\gamma$ is defined as:
\begin{equation}
\label{eq:defgamma}
\gamma := \frac{1}{c^2-1}.
\end{equation}


For all periodic traveling waves $f(z)$,  equation \eqref{eq:p} has periodic coefficients.  Let $T$, the \emph{fundamental period} of $f$, denote the smallest positive number for which $f(z+T)=f(z)\pmod{2\pi}$.

\begin{definition} \label{def:tempeigenP}We say that $\lambda\in\C$ is a {\em temporal eigenvalue} if there exists a solution to \eqref{eq:p} which is bounded for all $z \in \R$.  The set of all temporal eigenvalues is called the {\em spectrum} of equation \eqref{eq:p}, and is denoted $\sigma\eqref{eq:p}$.
\end{definition}

\begin{remark}\label{rem:eigen} We pause here for some brief remarks on the terminology used in Definition \ref{def:tempeigenP}, and to draw attention to some further remarks later on.

First we note that, as it is stated, Definition \ref{def:tempeigenP} in relation to equation \eqref{eq:p} is a nonstandard eigenvalue problem. Rather than a linear eigenvalue parameter, and an equation of the form $\cA v = \lambda v$, we have a quadratic operator pencil (see Remark \ref{rem:quadratic}). This leads us to take a more classical, applied mathematics approach to the problem. 

Secondly, it should be noted that our eigenfunctions are not in $L^2(\R)$ but rather $C_b(\R)$, the space of continuous and bounded functions on $\R$. Whether or not we can infer the full $L^2(\R)$ spectrum requires more work (see Remark \ref{rem:completestability}) and is key in moving from spectral stability analysis to full linear and nonlinear stability analysis. However, as the focus of the paper is the former, we fear going too far afield on this tangential matter. 

Finally, we remark that our eigenvalues as given in Definition \ref{def:tempeigenP} are not isolated. Given the periodicity of $\cos(f(z))$, we can approach this issue in the standard way; by following the construction of equation \eqref{eq:HillEV}, we see that the spectrum $\sigma\eqref{eq:p}$ decomposes into a one parameter family of (Bloch) eigenvalues each of which is isolated in the appropriate function space. Considering them as such was not necessary for any of the results in this paper, however such a construction is useful when 
\eqref{eq:sinegordon} is generalized to a nonlinear Klein-Gordon with arbitrary potential $V$ (see \cite{JMMP2}).
\end{remark}

\begin{definition} \label{def:stability}Let $f(z)$ be a periodic traveling wave solution of the sine-Gordon equation.   We say that $f(z)$ is {\em (temporally spectrally) stable} provided that there are no temporal eigenvalues in the right half plane. If there is a temporal eigenvalue in the right half plane, we say that $f(z)$ is {\em (temporally spectrally) unstable}.
\end{definition}

The spectrum $\sigmaP$ has the following four-fold symmetry.
\begin{proposition} \label{lem:symmetryp} 
$\sigmaP=\sigmaP^*=-\sigmaP=-\sigmaP^*$.
\end{proposition}
\begin{proof} Suppose $\lambda$ is such that there exists a bounded solution $p(z)$ to equation 
\eqref{eq:p}. By taking the complex conjugate of equation \eqref{eq:p} we see that $p(z)^*$ 
solves equation \eqref{eq:p} when $\lambda = \lambda^*$. Moreover $p(z)^*$ is bounded if $p(z)$ is for all $z \in \R$. Thus $\lambda^*$ is in the  spectrum. Likewise for $- \lambda$, we observe that for all types of periodic traveling waves under consideration, there exists some $z_0 \in \R$, such that $\cos(f(z - z_0)) = \cos(f(z_0-z))$. Without loss of generality 
(by choice of origin for $z$ in the autonomous equation \eqref{eq:pendulum}),
we 
may assume that $z_0 = 0$, that is, $\cos(f(z))$ is an even function of $z$. Thus, if $\lambda\in \sigmaP$, substitution of the bounded function $p(-z)$ into equation \eqref{eq:p} shows that $-\lambda$ is in the  spectrum as well. Combining these last two results completes the proof. 
\end{proof}
\begin{remark}
The significance of Proposition~\ref{lem:symmetryp} is that the spectral problem \eqref{eq:p} has
so-called \emph{full Hamiltonian symmetry}.  Of course this means that spectral instability corresponds to the existence of a temporal eigenvalue with nonzero (positive or negative) real part.  Proposition~\ref{lem:symmetryp} also holds for the spectrum corresponding to periodic traveling waves in more general Klein-Gordon type equations with potential $V$, and while the argument that
$\sigmaP^*=\sigmaP$ goes through unchanged, the proof of the fact that $\sigmaP=-\sigmaP$
is different in general because for rotational waves $V''(f(z))$ (which plays the role of $\cos(f(z))$ in the Klein-Gordon case) is not generally even about any $z_0$.  See \cite{JMMP2} for more details.
\end{remark}
\begin{remark} \label{rem:quadratic}
Under the substitution $\lambda=i\zeta$, equation \eqref{eq:p} can be written in terms of
a (formally) \emph{selfadjoint quadratic operator pencil} $\mathcal{L}(\zeta)$ given by 
$\mathcal{L}(\zeta):=\zeta^2\mathcal{L}_2 + \zeta\mathcal{L}_1 + \mathcal{L}_0$, where
\begin{equation}
\mathcal{L}_2:=-\gamma,\quad\mathcal{L}_1:= -2c\gamma\cdot i\frac{d}{dz},\quad
\mathcal{L}_0:=\frac{d^2}{dz^2}+\gamma\cos(f(z)).
\end{equation}
Indeed, it is easy to see that \eqref{eq:p} can be written simply as $\mathcal{L}(\zeta)p=0$.  In general, a spectral problem for a polynomial operator pencil can be 
reformulated as a genuine eigenvalue problem for an operator acting on an appropriate Cartesian product of the  base space.  Upon setting $p_0=p$ and $p_1=\lambda p$, 
we can easily rewrite \eqref{eq:p} in the form
\begin{equation}
 L_{f}
\begin{pmatrix} p_0 \\ p_1 \end{pmatrix}=\lambda\begin{pmatrix} p_0 \\ p_1 \end{pmatrix},\quad  L_{f}:=
\begin{pmatrix} 0 & 1 \\ -(c^2-1) \partial_{zz} - \cos (f(z)) & 2 c \partial_{z} \end{pmatrix}.
\label{eq:w-system}
\end{equation}
The matrix operator $L_f$ is frequently called the \emph{companion matrix} to the pencil $\mathcal{L}$.  In principle, these observations make available the theory of Krein signatures associated to the purely imaginary temporal eigenvalues, which can in turn be used to locate non-imaginary points of $\sigmaP$ corresponding to instability.  Indeed, an alternate proof of Lemma~\ref{lem:sub_rot} below can be given in terms of Krein signature theory.  See \cite{KollarM12} and the references therein for further information. 
\end{remark}
We are now ready to formulate our main result. 
We emphasize that stability is meant in the sense of Definition \ref{def:stability}.
\begin{theorem} \label{th:main}Subluminal rotational waves are stable. Superluminal rotational waves are unstable. Both sub- and superluminal librational waves are unstable. 
\end{theorem}
In Figure~\ref{fig:spec} we show numerically computed spectra for each of the four types of periodic traveling waves.
The structure of $\sigmaP$, including details of the behavior near the origin in the $\lambda$-plane will be described (also for a more general class of nonlinear Klein-Gordon type equations)
in a subsequent work \cite{JMMP2}.

\begin{figure}[h]
\begin{center}
\includegraphics{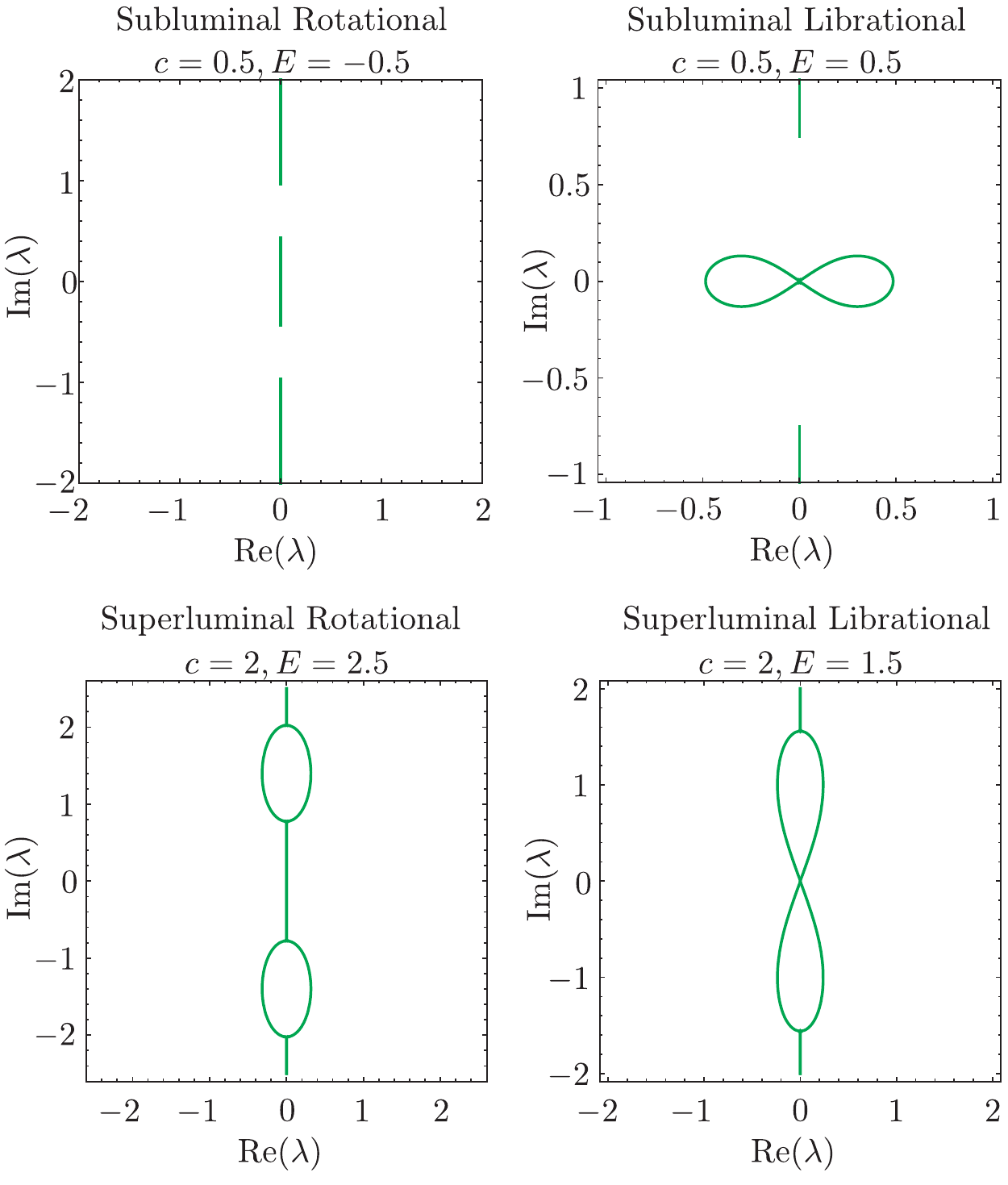}
\end{center}
\caption{Numerical plots of the spectrum $\sigma\eqref{eq:p}$ in the complex $\lambda$-plane representative of  the four types of periodic traveling waves.  Top row:  subluminal waves.  Bottom row:  superluminal waves.  Left column:  rotational waves.  Right column:  librational waves.  To compute the part of $\sigmaP$ on the imaginary axis it is useful to use Corollary~\ref{cor:pandqimag} below to reduce the problem to the calculation of the trace of the monodromy matrix for Hill's equation \eqref{eq:q}.  To compute the non-imaginary part of $\sigmaP$ one can look for the zero level curve of $G_p(\lambda)$ to be defined below (see \eqref{eq:GpGqdef}).
  }\label{fig:spec}
\end{figure}

\begin{remark}\label{rem:completestability}
The complete stability analysis of periodic traveling wave solutions of the sine-Gordon equation \eqref{eq:sinegordon} requires two additional nontrivial steps.  First, one requires some analogue of the spectral theorem guaranteeing \emph{completeness} of the set of nonzero solutions $p(z)$ of equation \eqref{eq:p} corresponding to $\lambda\in\sigmaP$.  This would allow the general solution $w(z,\tau)$ of the Cauchy initial-value problem for the linearized sine-Gordon equation \eqref{eq:linearpde} to be expressed uniquely as an integral over $\sigmaP$ of factorized solutions $w(z,\tau)=p(z)e^{\lambda\tau}$.  Second, one must prove that $w(z,\tau)$ is a good approximation to the solution $v(z,\tau)$ of \eqref{eq:travelingsinegordon} with the same initial data, i.e.\ one must prove \emph{nonlinear stability}.  However, the notion of spectral stability is the essential starting point for both of these next steps.

It should also be said that the kind of stability problem we have in mind from the start is the analysis of the effect of perturbations to initial data at $\tau=t=0$ that are small in an appropriate function space.  A quite different approach would be to exploit the invariance of the sine-Gordon equation \eqref{eq:sinegordon} under the Poincar\'e group of Lorentz boosts to make the traveling wave stationary.  This latter approach has the advantage of avoiding altogether the first-order derivatives in equation \eqref{eq:p} arising from the use of a Galilean boost (under which \eqref{eq:sinegordon} is \emph{not} invariant) to make the wave stationary; however this method only applies for subluminal traveling waves, and moreover physically one is then perturbing the traveling wave on a space-like hypersurface different from the line $t=0$.
\end{remark}

To our knowledge, Scott \cite{Sco1} was the first to
consider the stability of periodic traveling wave solutions to the sine-Gordon equation.
His final conclusion of which cases are stable, as well as which are not, turns out to have been correct; however, it was based on a claim that we show here was not correct. Scott claimed that the spectrum was, in every case, contained in the union of the real or imaginary axes. This is unfortunately not correct, as we will show, and it is the fact the spectrum cannot be so easily constrained that makes this a difficult problem. 
Scott made the observation that equation \eqref{eq:p} can be converted by a simple exponential substitution
\begin{equation}\label{eq:defineq}
q(z) =p(z) e^{- c \gamma  \lambda z},
\end{equation}
into \emph{Hill's equation} with spectral parameter $\mu$: 
\begin{equation*}\label{eq:q}\tag{{\bf Q}}
q'' + \gamma \cos(f(z))q =\mu q, \quad \mu := \gamma^2 \lambda^2 = \left(\frac{\lambda}{c^2-1}\right)^2.
\end{equation*}
This method is advantageous in that Hill's equation has been  well-studied, and there is a large body of related work dating to the late 19th century (see, for example, \cite{MW66} and the references therein).  Using this connection and information about the spectrum of Hill's equation, Scott attempted to deduce the location of $\sigmaP$ in the complex plane.
Unfortunately, the argument in \cite{Sco1} assumes that the transformation \eqref{eq:defineq}
is isospectral, i.e.\ that the exponential factor $e^{-c\gamma\lambda z}$ is harmless.  As the spectral problem for equation \eqref{eq:q} was posed on $C_b(\R)$, 
this is not generally true unless $\lambda$ is purely imaginary (as we will prove below; see Corollary~\ref{cor:pandqimag} and Lemma~\ref{lem:pandqimag}).  This led to an erroneous calculation of $\sigmaP$ in \cite{Sco1}. 
The \emph{statement} of Theorem~\ref{th:main} is in fact rather well-known in the nonlinear waves community (see, for example the influential textbook by Whitham \cite{Wh} which cites \cite{Sco1}), however 
to our knowledge there is to date no correct proof in the literature.  Our goal in  this paper is to give a completely rigorous proof of Theorem~\ref{th:main}, properly accounting for the effect of the exponential factor in the transformation \eqref{eq:defineq}.

The rest of our paper is organized as follows.  In \S\ref{sec:eigenvalue} we define the basic notions of spectral analysis for the problems \eqref{eq:p} and \eqref{eq:q}.  Then in \S\ref{sec:fundamental_properties} we present some fundamental relations between the spectra of problems \eqref{eq:p} and \eqref{eq:q} and we also recall the most important properties of
Floquet theory for Hill's equation.   At this point we will have all of the tools in place to give the proof of Theorem~\ref{th:main}, all details of which are presented in \S\ref{sec:stability_instability}.
Finally, in \S\ref{sec:structure} we give some auxiliary results describing aspects of the structure
of the spectrum of equation \eqref{eq:p} in the complex $\lambda$-plane.  In particular these results show that the spectrum of \eqref{eq:p} is not necessarily confined to the real and imaginary axes as was suggested in the paper \cite{Sco1}. For the reader's convenience, in an appendix we gather some results about how problem \eqref{eq:q} can be reduced to Lam\'e's equation and implications for the corresponding Floquet spectrum.

\section{The Eigenvalue Problem} 
\label{sec:eigenvalue}
%
%
\subsection{Floquet theory for equations \eqref{eq:p} and \eqref{eq:q}}
Consider a general first-order $2\times 2$ system with periodic coefficients, of the form
\begin{equation}
\F'(z;\lambda)=\A(z;\lambda)\F(z;\lambda),\quad \A(z+T;\lambda)=\A(z;\lambda).
\label{eq:periodic_general}
\end{equation}
Here, the matrix $\A(z;\lambda)$ is entire in $\lambda$ for each $z\in\mathbb{R}$.
The Floquet theory of this equation is based on the analysis of the fundamental solution matrix $\F(z;\lambda)$ satisfying \eqref{eq:periodic_general} and the initial condition $\F(0;\lambda):=\Id$.  The relevant quantities are defined as follows.
\begin{definition}
The matrix $\M(\lambda):=\F(T;\lambda)$ is called the \emph{monodromy matrix} for equation \eqref{eq:periodic_general}.  We denote the trace and determinant of the monodromy matrix by
$\Delta(\lambda):=\M_{11}(\lambda)+\M_{22}(\lambda)$ and $D(\lambda):=\M_{11}(\lambda)\M_{22}(\lambda)-\M_{12}(\lambda)\M_{21}(\lambda)$ respectively.  The \emph{discriminant} is the quantity $\Delta(\lambda)^2-4D(\lambda)$.  \label{def:Monodromy}
\end{definition}
\noindent Frequently, Abel's identity can be used to calculate a closed form for$D(\lambda)$: 
\begin{equation}\label{eq:Abel}
D(\lambda) = \exp\left(\int_0^T \tr(\A(z;\lambda))\, dz\right).
\end{equation}

Equation \eqref{eq:p} can be written in the form of equation \eqref{eq:periodic_general}:
\begin{equation}\label{eq:psystem}
\begin{pmatrix} p \\ p' \end{pmatrix}' =\A_p(z; \lambda) \begin{pmatrix} p \\p' \end{pmatrix},\quad
\A_p(z;\lambda):= \begin{pmatrix} 0 & 1 \\ -\gamma (\lambda^2 + \cos(f(z))) & 2 c \gamma \lambda\end{pmatrix}.
\end{equation}
Note that since $f(z)$ is periodic modulo $2\pi$, $\cos(f(z))$ is also periodic.  We denote the  quantities defined in Definition~\ref{def:Monodromy} corresponding to equation \eqref{eq:p} or equivalently \eqref{eq:psystem} using the subscript ``$p$''.  
\begin{definition}
The \emph{Floquet multipliers} $\rho=\rho(\lambda)$ of equation \eqref{eq:p} are the (spatial) eigenvalues of the monodromy matrix $\M_p(\lambda)$, i.e., they solve the quadratic equation $\rho^2-\Delta_p(\lambda)\rho +D_p(\lambda)=0$.
\label{def:Floquet_P}
\end{definition}

\noindent The next proposition is a standard result from Floquet theory (e.g.\ see \cite{Chi99}): 
\begin{proposition}
The complex number $\lambda$ is a temporal eigenvalue, i.e.\ belongs to the spectrum $\sigma\eqref{eq:p}$, if and only if at least one of the Floquet multipliers $\rho$ has modulus $1$:  $|\rho|=1$.  
\label{prop:FloquetSpectrumP}
\end{proposition}
\noindent For this reason we sometimes refer to $\sigmaP$ as the \emph{Floquet spectrum} (of \eqref{eq:p}).


The relation between equations \eqref{eq:p} and \eqref{eq:q} given by \eqref{eq:defineq} allows us to reinterpret the spectrum $\sigmaP$ as originally defined in Definition \ref{def:tempeigenP}:
\begin{proposition}
The Floquet spectrum $\sigmaP$ consists of exactly those $\lambda\in\C$ such that there is a solution $q(z)$ to \eqref{eq:q} for which  $q(z)e^{c \gamma \lambda z}$ is bounded for all $z\in\R$.
\end{proposition}

Just as for \eqref{eq:p} we can write \eqref{eq:q} as a system with periodic coefficients:\begin{equation}\label{eq:qsystem}
\begin{pmatrix} q \\ q' \end{pmatrix}' =   \A_q(z;\lambda)\begin{pmatrix} q \\ q' \end{pmatrix} \quad \A_q(z;\lambda) := 
\begin{pmatrix} 0 & 1 \\  \mu -\gamma \cos(f(z)) & 0 \end{pmatrix}.
\end{equation}
We denote the  quantities defined in Definition~\ref{def:Monodromy} corresponding to equation \eqref{eq:q} or equivalently \eqref{eq:qsystem} using the subscript ``$q$''.  
\begin{definition}
The \emph{Floquet multipliers} $\eta=\eta(\lambda)$ of equation \eqref{eq:q} are the (spatial) eigenvalues of the monodromy matrix $\M_q(\lambda)$, i.e., they solve the quadratic equation $\eta^2-\Delta_q(\lambda)\eta +D_q(\lambda)=0$.
\label{def:Floquet_Q}
\end{definition}
The definition parallel to Definition~\ref{def:tempeigenP} but tailored to  \eqref{eq:q} is the following.
\begin{definition} The \emph{(Floquet) spectrum} $\sigma\eqref{eq:q}$ of equation \eqref{eq:q} is the set of $\lambda \in \C$ such that there exists a bounded (on $\R$) solution $q(z)$ to equation \eqref{eq:q} with $\mu = (\gamma \lambda)^2$. Equivalently, $\lambda \in \sigma\eqref{eq:q}$ if and only if at least one of the Floquet multipliers $\eta$ has modulus $1$: $|\eta| = 1$. 
\label{def:tempeigenQ}
\end{definition}

\noindent From equations \eqref{eq:psystem} and \eqref{eq:qsystem} and formula \eqref{eq:Abel}  we have for all $\lambda \in \C$ 
\begin{equation}\label{eq:determinants}
D_p(\lambda) = e^{2c\gamma\lambda T} \quad \text{and} \quad D_q(\lambda) = 1.
\end{equation} 

\begin{remark}
We stress here that in general $\sigmaP$ is {\em not} the same as $\sigmaQ$.  It is an elementary fact from Hill's equation theory \cite{MW66} that the spectral problem \eqref{eq:q} is self-adjoint, and hence $\sigmaQ$ is confined to the real and imaginary axes, as $\mu$ must be real in order for there to exist a bounded solution to equation \eqref{eq:q}. However, it is evident from Figure \ref{fig:spec} that $\sigmaP$ can have points that are neither purely real nor purely imaginary. Moreover, $\lambda \in \sigmaP$ corresponds to the existence of a solution to \eqref{eq:q} exhibiting a certain growth rate for $z \in \R$, and it is important to keep in mind 
that in determining stability of traveling wave solutions to the sine-Gordon equation, it is not the Floquet 
spectrum of \eqref{eq:q} per se that will explicitly indicate stability, as the latter corresponds to solutions of \eqref{eq:q} bounded for $z \in \R$. 

On the other hand, the Floquet spectra of \eqref{eq:p} and \eqref{eq:q} \emph{are} related as we will show in \S \ref{subsec:relate}.
\end{remark}

\begin{remark}
It is not in general true that the fundamental period $T$ of $f$ will coincide with the smallest period of equation \eqref{eq:psystem}. For rotational waves they do in fact coincide, but for librational waves it is straightforward to see that $T$ is actually {\em twice} the minimal period of $\cos (f(z))$. 
In computing the monodromy matrix $\M(\lambda)$ for equations \eqref{eq:p} or \eqref{eq:q}, we
will \emph{always} use the fundamental period $T$ of $f$, regardless of whether or not it is also the minimal period of \eqref{eq:p} or \eqref{eq:q}.
\end{remark}

\begin{remark}
Note that the function $\Delta_q(\lambda)$ as given in Definition \ref{def:Monodromy} is frequently called the \emph{Hill discriminant} or sometimes just the \emph{discriminant} \cite{MW66}. We wish to emphasize that as such, this is {\em not} the discriminant $\Delta_q(\lambda)^2 - 4D_q(\lambda)$ of the characteristic polynomial of the monodromy matrix $\M_q(\lambda)$. The reason for our terminology is that the discriminant of the characteristic polynomial of $\M_p(\lambda)$ plays an important role in the analysis that follows but unlike that of $\M_q(\lambda)$ it is non-trivially related to the corresponding trace, and we wish to use parallel language for equations \eqref{eq:p} and \eqref{eq:q}. 
\end{remark}

\section{Fundamental Properties of the Floquet Spectra}
\label{sec:fundamental_properties}
The purpose of introducing Hill's equation \eqref{eq:q} is that much is known about it and its Floquet spectrum. An aim of this section is to state and prove those results which give us some information about the structure of the Floquet spectra of equations \eqref{eq:p} and \eqref{eq:q}, but which are not dependent on knowing the type of the periodic traveling wave $f(z)$  (i.e.\ sub- or superluminal, rotational or librational). 

\subsection{Relating the spectra of equations \eqref{eq:p} and \eqref{eq:q}} \label{subsec:relate}

We will begin by relating the Floquet multipliers of \eqref{eq:p} and \eqref{eq:q}.
\begin{lemma}\label{lem:floquetpandq} The values $\{\rho_+(\lambda),\rho_-(\lambda)\}$ are the Floquet multipliers of \eqref{eq:p} if and only if  the corresponding values $\{\eta_+(\lambda),\eta_-(\lambda)\}:= \{e^{ - c \gamma \lambda T}\rho_+(\lambda),e^{-c\gamma\lambda T}\rho_-(\lambda)\}$ are the Floquet multipliers of \eqref{eq:q}.
\end{lemma}

\begin{proof} Let $p$ and $q$ be solutions to \eqref{eq:p} and \eqref{eq:q} respectively, related by equation~\eqref{eq:defineq}: $q = p e^{- c \gamma \lambda z}$. We have: 
\begin{equation}
 \begin{pmatrix} q \\ q'  \end{pmatrix}   = e^{-c \gamma \lambda  z} \Hm(\lambda)\begin{pmatrix} p \\ p' \end{pmatrix} \quad \text{where} \quad \Hm(\lambda) := \begin{pmatrix} 1 & 0 \\ -c \gamma \lambda & 1 \end{pmatrix}.
\end{equation}
Recall the fundamental solution matrices $\F_{p,q}(z;\lambda)$, solving equations \eqref{eq:psystem} or \eqref{eq:qsystem} with the initial conditions $\F_{p,q}(0;\lambda) = \Id$. We let $\Q(z;\lambda)$ be the matrix of functions given by 
\begin{equation}
\Q(z;\lambda) := e^{- c \gamma \lambda z} \Hm(\lambda) \F_p(z;\lambda).
\end{equation}
We observe that $\Q(z;\lambda)$ is a matrix solution to \eqref{eq:q} satisfying $\Q(0;\lambda) = \Hm(\lambda)$. This means that we can write 
\begin{equation}
\Q(z;\lambda) = \F_q(z;\lambda)\Hm(\lambda).
\end{equation}
Equating these last two expressions and evaluating at $z = T$, and using the fact that $\det(\Hm(\lambda)) = 1$, we can then write
\begin{equation}\label{eq:proof1}
\M_q(\lambda) = e^{- c \gamma \lambda T} \Hm(\lambda) \M_p(\lambda) \Hm(\lambda)^{-1}.
\end{equation}
The eigenvalues $\{\eta_+(\lambda),\eta_-(\lambda)\}$ of the left-hand side (Floquet multipliers of \eqref{eq:q}) must therefore equal the eigenvalues $\{e^{-c\gamma\lambda T}\rho_+(\lambda),e^{-c\gamma\lambda T}\rho_-(\lambda)\}$ of the right-hand side (where $\{\rho_+(\lambda),\rho_-(\lambda)\}$ are the Floquet multipliers of \eqref{eq:p}).
\end{proof}

\begin{corollary}\label{cor:pandqimag} The imaginary number $\lambda = i \beta$,  $\beta \in \R$, is in the Floquet spectrum of \eqref{eq:p} if and only if it is in the Floquet spectrum of \eqref{eq:q}. 
\end{corollary}

The  Floquet spectra of \pandq\  therefore agree on the imaginary axis. The next lemma says that it is the only place where this occurs.  

\begin{lemma} \label{lem:pandqimag}Suppose that $\lambda \in \sigmaP \cap \sigmaQ$. Then $\lambda$ is purely imaginary. 
\end{lemma}

\begin{proof} Suppose that $\lambda $ is in the Floquet spectrum of \eqref{eq:q}. Then in particular every nonzero solution of \eqref{eq:q} is bounded away from zero infinitely often as $|z|\to\infty$,  $z\in\R$.   Equation \eqref{eq:defineq} then implies that every non-zero solution to \eqref{eq:p} is exponentially unbounded, and hence $\lambda \notin \sigmaP$, unless $\lambda$ is purely imaginary. 
\end{proof}

We now illustrate some of the main differences between $\sigmaP$ and $\sigmaQ$. 
It follows directly from the fact that $D_q(\lambda) = 1$ that if $\lambda\in\sigmaQ$, then both Floquet multipliers of \eqref{eq:q} lie on the unit circle.  However for \eqref{eq:p} we have instead the following.
\begin{lemma}\label{lem:pabel} 
Assume that $c\neq 0$ (in addition to $c^2\neq 1$).
Let $\lambda \in \sigmaP$ and suppose that $\re{\lambda} \neq 0$. Then one and only one of the Floquet multipliers of \eqref{eq:p} will lie on the unit circle. 
\end{lemma}

\begin{proof} Let $\rho_\pm(\lambda)$ denote the Floquet multipliers of equation \eqref{eq:p}. Equation \eqref{eq:determinants} implies that $|\rho_+(\lambda)|\cdot|\rho_-(\lambda)| = e^{2c\gamma T \re{\lambda}}$. If $\lambda \in \sigmaP$, then one of the factors on the left hand side is equal to 1. But if $\re{\lambda} \neq 0$, then the other factor cannot also be equal to 1.
\end{proof}

\begin{lemma}\label{lem:realp} 
Assume that $c\neq 0$ (in addition to $c^2\neq 1$).
Suppose $\lambda \in \R \setminus \{0\}$ is in the Floquet spectrum of \eqref{eq:p}. Then one of the Floquet multipliers of \eqref{eq:p} satisfies $\rho^2=1$,
i.e., $\lambda$ is either a periodic ($\rho=1$) or antiperiodic ($\rho=-1$) temporal eigenvalue.
\end{lemma}
\noindent Note that by contrast if $\lambda\in \sigmaQ$ is real and nonzero it is not in general a periodic or antiperiodic Floquet eigenvalue of Hill's equation \eqref{eq:q}. 

\begin{proof} Equation~\eqref{eq:determinants} implies that the product of the Floquet multipliers of \eqref{eq:p} is 
positive, and is explicitly given by $e^{2 c \gamma \lambda T}\neq 1$. Since $\lambda\in\R$, the Floquet multipliers are either real or form a complex-conjugate pair.  The latter case can be ruled out because for $\lambda\in\sigmaP$
the product of the (necessarily unimodular) multipliers would equal $1$.  So the multipliers are
both real and since $\lambda\in\sigmaP$ one of them must be $\pm 1$, in which case the other is given by $\pm e^{2c\gamma\lambda T}\neq\pm 1$.
\end{proof}

We take a moment here to highlight some of the consequences for the Floquet spectra of equations \pandq\ 
that are due to the past few results. 

\begin{proposition} \label{cor:sum} Let $\lambda$ be a point where the discriminant of \eqref{eq:p} vanishes, i.e.\ the Floquet multipliers are equal. Denote the equal values by $\rho$. Then 
\begin{enumerate}
\item The Floquet multipliers of \eqref{eq:q} are also equal. Denoting the equal values by $\eta$, we have that $\eta = \pm 1$ and correspondingly $\rho = \pm e^{c \gamma \lambda T}$.
\item $\lambda$ is either real or purely  imaginary. If it is real and non-zero, then $\lambda \notin \sigmaP$, while if it is purely imaginary (including zero), then $\lambda \in \sigmaP$.  
\end{enumerate}
\end{proposition}

\begin{proof} The fact that the Floquet multipliers of \eqref{eq:q} are equal follows from Lemma~\ref{lem:floquetpandq}. Formula~\eqref{eq:determinants} implies that $\eta^2 = D_q(\lambda) = 1$, hence $\eta = \pm 1$. Applying Lemma~\ref{lem:floquetpandq} again implies that $\rho = \pm e^{c \gamma \lambda T}$. Since $|\eta| = 1$, we have $\lambda \in \sigmaQ$, and by self-adjointness of \eqref{eq:q}, this means that $\im{\lambda^2} = 0$. Corollary~\ref{cor:pandqimag} and Lemma~\ref{lem:pandqimag} together imply the dichotomy in statement (2). 
\end{proof}

\subsection{The spectrum of equation \eqref{eq:q}} \label{sec:hillspec}

A general reference for the Hill's equation theory in this section is Magnus and Winkler \cite{MW66}.  Consider the Hill's operator
\begin{equation}\label{eq:hillop}
L:=\frac{d^2}{dz^2}+\gamma\cos(f(z)),
\end{equation}
where $f(z)$ is the $T$-periodic (mod $2\pi$) traveling wave solution of the sine-Gordon equation about which we are linearizing. The spectrum $\sigmaQ$ is characterized by nonzero solutions $q$ of
\begin{equation}
Lq(z)=\mu^{(\theta)} q(z),\quad q(z+T)=e^{i\theta}q(z),
\label{eq:HillEV}
\end{equation}
where $\theta$ is a \emph{real} phase angle.  (Recall the relation $\mu=\gamma^2\lambda^2$; in this section we will consider $\mu$ rather than $\lambda$ as the key spectral parameter for Hill's equation \eqref{eq:q}.)

The numbers $\mu^{(0)}$ correspond to the periodic eigenvalues of \eqref{eq:q}, while the numbers
$\mu^{(\pi)}$ correspond to the antiperiodic eigenvalues.
The set of $\mu^{(\theta)}\in\mathbb{R}$ for which there exists a nontrivial solution of \eqref{eq:HillEV} will be denoted $\Sigma_\theta({\bf Q})$.  Since $\cos (f(z))$ is real, it is clear that $\Sigma_{-\theta}({\bf Q})=\Sigma_\theta({\bf Q})$.  The Floquet spectrum $\sigmaQ$ is related to the union of these sets over $\theta$ as follows:
\begin{equation}
\gamma^2\sigmaQ^2=\Sigma({\bf Q}):=\bigcup_{-\pi<\theta\le\pi}\Sigma_\theta({\bf Q}).
\end{equation}
The $\mu$-Floquet spectrum $\Sigma({\bf Q})$ of Hill's operator \eqref{eq:hillop} is bounded above.  It consists of the union of closed intervals
\begin{equation}
\Sigma({\bf Q})=\bigcup_{n=0}^\infty  [\mu^{(0)}_{2n+1},\mu_{2n+2}^{(\pi)}]\cup[\mu^{(\pi)}_{2n+1},\mu_{2n}^{(0)}]
\end{equation}
where the sequences $\Sigma_0({\bf Q}):=\{\mu^{(0)}_{j}\}$ and $\Sigma_\pi({\bf Q}):=\{\mu^{(\pi)}_{j}\}$ decrease to $-\infty$ and satisfy the inequalities
\begin{equation}
\cdots<\mu_4^{(\pi)}\le\mu_3^{(\pi)}<\mu_2^{(0)}\le\mu_1^{(0)}<\mu_2^{(\pi)}\le\mu_1^{(\pi)}<\mu_0^{(0)}.
\end{equation}

The function $f'(z)$ is a nontrivial $T$-periodic solution ($\theta=0$) when $\mu=0$, that is, $Lf'(z)=0$, and $f'(z+T)=f'(z)$.  Hence one of the periodic eigenvalues $\mu^{(0)}_{j}$ coincides with $\mu=0$, and the value of $j$ is determined by oscillation theory (see  \cite[Thm.\ 2.14 (Haupt's Theorem)]{MW66}\footnote{The statement of Haupt's theorem in \cite{MW66} contains some typographical errors; the second sentence of Theorem~2.14 should be corrected to read ``If $\lambda=\lambda'_{2n-1}$ or $\lambda=\lambda'_{2n}$, then $y$ has exactly $2n-1$ zeros in the half-open interval $0\le x<\pi$.''} or   \cite[Thm.\ 8.3.1]{CL55}).  For librational $f$, $f'(z)$ has exactly two zeros per period and hence either $\mu^{(0)}_{1}=0$ or $\mu^{(0)}_2=0$.  On the other hand, for rotational $f$, $f'(z)$ has no zeros at all and hence $\mu^{(0)}_{0}=0$. 
Therefore, we have the following dichotomy for the spectrum of Hill's equation:
\begin{itemize}
\item
For librational $f$, the positive part of the $\mu$-Floquet spectrum $\Sigma({\bf Q})$ consists of the intervals $[\mu^{(0)}_1,\mu^{(\pi)}_2]\cup[\mu^{(\pi)}_1,\mu_0^{(0)}]$ (which may merge into a single interval if $\mu^{(\pi)}_1=\mu^{(\pi)}_2$).  It is possible that $\mu^{(0)}_1=0$, but not necessary; otherwise $\mu^{(0)}_2=0$ and $\mu^{(0)}_1>0$.  
\item 
For rotational $f$, the $\mu$-Floquet spectrum $\Sigma({\bf Q})$ is a subset of the closed negative $\mu$ half-line, and $\mu^{(0)}_0=0$ belongs to the spectrum.
\end{itemize}
In both cases the $\mu$-Floquet spectrum has an unbounded negative part and consists of bands separated by gaps. The bands correspond to $\left| \Delta_q \right| \leq 2$, and the gaps correspond to $ \left| \Delta_q \right| > 2$. 
Whether or not $\mu_1^{(0)}=0$ for librational waves can be resolved as follows.
\begin{lemma}\label{lem:mu1} Let $f(z)$ be a librational traveling wave of the sine-Gordon equation.  Then $0=\mu_1^{(0)}> \mu_2^{(0)}$.
\end{lemma}

\begin{proof}
We begin by characterizing the dependence of the fundamental period $T$ on $c$ and $E$.  Since $f(z)$ is a librational traveling wave,  we have $0<E<2$.  By integrating \eqref{eq:pendint} along the appropriate orbits inside the separatrix, we obtain:
\begin{equation}
T=\begin{cases} \sqrt{2(c^2-1)}P(E),&\quad c^2>1\\\\
\sqrt{2(1-c^2)}P(2-E),&\quad c^2<1,
\end{cases}
\label{eq:period}
\end{equation}
where
\begin{equation}
P(E):= 2\int_0^{\arccos(1-E)}\frac{df}{\sqrt{E-1+\cos(f)}}.
\end{equation}
(The subluminal period formula follows from the superluminal period formula by the substitution $f\mapsto f+\pi$.)

Now, since $\lambda=0$ is a periodic eigenvalue of \eqref{eq:q}, we have $\Delta_q=2$ 
when $\mu=0$ and in this case there are two alternatives:
\begin{equation}
\mu_1^{(0)}=0 > \mu_2^{(0)} \;\Leftrightarrow\; \left.\dfrac{d}{d\mu}\Delta_q\right|_{\mu=0}<0
\quad\text{or}\quad
\mu_2^{(0)} = 0 \leq \mu_1^{(0)}\;\Leftrightarrow\; \left.\dfrac{d}{d\mu}\Delta_q\right|_{\mu=0}\geq0,
\end{equation}
so it is enough to show that
\begin{equation}
\left.\frac{d}{d\mu}\Delta_q\right|_{\mu=0}<0.
\label{eq:deltaprimezero}
\end{equation}
%
%
%

In order to establish \eqref{eq:deltaprimezero}, we will first show that the monodromy matrix of equation \eqref{eq:q} takes the form: 
\begin{equation}
\M_q(0)= \begin{pmatrix} 1 & -v_0(E,c)^2 (c^2-1)T_E \\ 0 & 1 \end{pmatrix},
\label{eq:Mqzero}
\end{equation}
where $v_0(E,c)$ is a real non-zero constant related to the function $f(z)$ that will be  defined below, and where $T_E=\partial T/\partial E$.


To prove \eqref{eq:Mqzero}, fix a wave speed $c\neq \pm 1$ and a total energy $E\in (0,2)$.  Let $f(z;E,c)$ be the periodic wave uniquely determined modulo $2\pi$ from equation \eqref{eq:pendulum}, having total energy $E$, and for which 
\begin{equation}
f_z(0;E,c)>0\quad\text{and}\quad \sin(f(0;E,c))=0.
\label{eq:sinIC}
\end{equation}
We claim that 
when $\mu=0$
the two-dimensional vector space of solutions to \eqref{eq:qsystem} is spanned by the following (subscripts denote partial derivatives): 
\begin{equation}
\begin{pmatrix} f _ z \\ f_{zz} \end{pmatrix} \quad \textrm{ and } \quad \begin{pmatrix} f_{E} \\ f_{Ez} \end{pmatrix}
\end{equation}
To see this, first differentiate equation \eqref{eq:pendulum} with respect to either $z$ or $E$ to see that both vectors solve equation \eqref{eq:qsystem} 
for $\mu=0$.
Then to prove independence,  differentiate equation \eqref{eq:pendint} with respect to $E$ to get 
\begin{equation}
(c^2-1) f_z f_{zE} = 1 -f_E \sin( f ).
\end{equation}
\noindent Using this and \eqref{eq:pendulum} gives 
\begin{equation}\label{eq:detq}
\begin{vmatrix} f_z & f_E \\ f_{zz} & f_{Ez} \end{vmatrix} = f_zf_{Ez} - f_E f_zz = \gamma = \frac{1}{c^2-1} \neq 0 .
\end{equation}
\noindent We may therefore define a fundamental solution matrix at $\mu = 0$ by: 
\begin{equation}
\Q(z,0) = \begin{pmatrix} f_z(z;E,c) & f_E(z;E,c) \\ f_{zz}(z;E,c) & f_{Ez}(z;E,c) \end{pmatrix}. 
\end{equation}
Of course the unique fundamental solution matrix $\F_q(z,0)$ satisfying $\F_q(0,0)=\Id$ can then be expressed as $\F_q(z,0)=\Q(z,0)\Q(0,0)^{-1}$.
In particular, by setting $z=T$ we can thus write the monodromy matrix at $\mu = 0$ as 
\begin{equation}
\M_q(0)=\F_q(T,0) = \Q(T,0) \Q(0,0)^{-1}.
\label{eq:Mq0-1}
\end{equation}
Next we wish to explicitly relate $\Q(T,0)$ with $\Q(0,0)$. 
Denote the initial values of $f$ and $f_z$ as follows:
\begin{equation}
\begin{split}
u_0(E,c) &  := f(0;E,c) = f(T;E,c), \label{eq:ic}  \\ 
v_0(E,c)  & := f_z(0;E,c) = f_z(T;E,c).
\end{split}
\end{equation}  
Differentiate the initial conditions \eqref{eq:ic} with respect to $E$ to get that 
\begin{equation}
\begin{split}
\partial_E u_0 & = f_E(T;E,c) + T_Ef_z(T;E,c)  \\ 
& =  f_E(T;E,c) + T_Ef_z(0;E,c) \\ 
& = f_E(T;E,c) + T_E v_0(E,c) 
\end{split}
\end{equation}
\noindent and 
\begin{align}
f_{Ez}(0;E,c) = \partial_E v_0 (E,c) & = T_E f_{zz}(T;E,c) + f_{E z}(T;E,c).
\end{align}
Using periodicity of $f$ and $\sin(f)$, equation \eqref{eq:pendulum} and the initial conditions \eqref{eq:sinIC} imply that 
\begin{align}
f_{zz} (T;E,c) = - \gamma \sin(f(0;E,c))=0.
\end{align}
\noindent Combining these observations,        we conclude: 
\begin{equation}
\label{eq:Q0}
\begin{split}
\Q(0,0) &= \begin{pmatrix}
           f_z(0;E,c) & f_E(0;E,c) \\ f_{zz}(0;E,c) & f_{Ez}(0;E,c)
          \end{pmatrix} \\ &= \begin{pmatrix}
           v_0(E,c) & \partial_E u_0(E,c) \\ 0 & \partial_E v_0(E,c)
          \end{pmatrix},
\end{split}
\end{equation}
and likewise,
\begin{equation}
\begin{split}
\Q(T,0) &= \begin{pmatrix}
           f_z(T;E,c) & f_E(T;E,c) \\ f_{zz}(T;E,c) & f_{Ez}(T;E,c)
          \end{pmatrix} \\
&= \begin{pmatrix}
           v_0(E,c) & \partial_E u_0(E,c) - T_E v_0(E,c)\\ 0 & \partial_E v_0(E,c)           \end{pmatrix}\\
&= \Q(0,0) + \begin{pmatrix}
              0 & -T_E v_0(E,c) \\ 0 & 0
             \end{pmatrix}.
\end{split}
\end{equation}
Using these in \eqref{eq:Mq0-1} along with $\det(\Q(0,0))=\gamma$ yields \eqref{eq:Mqzero}.

Now, differentiation\footnote{This part of the argument follows closely \cite[pp. 15--18]{MW66}.} of the equation $\F_q'(z,\lambda)=\A_q(z,\lambda)\F_q(z,\lambda)$ (cf.\ equation \eqref{eq:qsystem}) with respect to $\mu=\gamma^2\lambda^2$ gives
\begin{equation}
\F_{q,\mu}'(z,\lambda)-\A_q(z,\lambda)\F_{q,\mu}(z,\lambda) = \A_{q,\mu}(z,\lambda)\F_q(z,\lambda)=\begin{pmatrix}0 & 0\\1 & 0\end{pmatrix}\F_q(z,\lambda).
\end{equation}
The general solution of this equation is obtained by variation of parameters, that is, by the substitution $\F_{q,\mu}(z,\lambda)=\F_q(z,\lambda)\W(z,\lambda)$ for some new unknown matrix $\W(z,\lambda)$.  Note that the initial condition $\F_q(0,\lambda)=\Id$, $\forall\lambda$ implies that $\W(0,\lambda)=0$.  It follows easily (also using $\det(\F_q(z,\lambda))=1$) that
\begin{equation}
\F_{q,\mu}(z,\lambda)=\F_q(z,\lambda)\int_0^z\begin{pmatrix}-\F_{q,11}(\zeta,\lambda)\F_{q,12}(\zeta,\lambda) & -\F_{q,12}(\zeta,\lambda)^2\\
\F_{q,11}(\zeta,\lambda)^2 & \F_{q,11}(\zeta,\lambda)\F_{q,12}(\zeta,\lambda)
\end{pmatrix}\,d\zeta.
\end{equation}
Setting $z=T$ and taking the trace gives
\begin{multline}
\frac{d}{d\mu}\Delta_q=(\M_{q,22}(\lambda)-\M_{q,11}(\lambda))\int_0^T\F_{q,11}(\zeta,\lambda)\F_{q,12}(\zeta,\lambda)\,d\zeta\\
{} -\M_{q,21}(\lambda)\int_0^T\F_{q,12}(\zeta,\lambda)^2\,d\zeta +\M_{q,12}(\lambda)\int_0^T\F_{q,11}(\zeta,\lambda)^2\,d\zeta.
\end{multline}
Upon taking $\lambda=0$ (and hence $\mu=0$) and using \eqref{eq:Mqzero} we see that
\begin{equation}
\left.\frac{d}{d\mu}\Delta_q\right|_{\mu=0}=-v_0(E,c)^2(c^2-1)T_E\int_0^T\F_{q,11}(\zeta,0)^2\,d\zeta.
\end{equation}
Because $\F_q(z,\lambda)$ is real-valued for all $\mu\in\mathbb{R}$ and $\F_{q,11}(z,0)$ does not vanish identically for $0<z<T$ because $\F_{q,11}(0,0)=1$, we easily conclude that
\begin{equation}\label{eq:keyid}
\sgn\left(\left.\frac{d}{d\mu}\Delta_q\right|_{\mu=0}\right)=-\sgn\left((c^2-1)T_E\right),
\end{equation}
where we have used the fact that $v_0(E,c)\neq 0$.


Finally,
we show that 
the inequality
\begin{equation}
 (c^2-1)T_E > 0 
 \label{eq:TEnegative}
 \end{equation}
holds strictly for {\em all} librational waves regardless of speed.
The proof is by direct computation of $T_E$. Recalling \eqref{eq:period}, we compute $P'(E)$ as follows:  since $0<E<2$, and hence $\arccos(1-E)<\pi$, we can make the substitution $w=\cos(f)$ implying $dw=-\sin(f)\,df=-\sqrt{1-w^2}\,df$ and hence obtain:
\begin{equation}
P(E)=
  2\int^1_{1-E} \frac{d w}{\sqrt{(1-w^2)(E-1+w)}}. 
 \end{equation}
 \noindent Next, making the substitution $ w = E(u-1) + 1$, so that $dw = E\, du$, yields: 
 \begin{equation}
P(E) =
2\int_0^1 \frac{du}{\sqrt{-E(u-1)^2 -2(u-1)}\sqrt{u}}.
 \end{equation}
 \noindent It is now easy to take the derivative with respect to $E$: 
 \begin{equation}
 P'(E)=
   \int_0^1\frac{(u-1)^2 du}{(-E(u-1)^2 - 2(u-1))^{\frac{3}{2}} \sqrt{u}} > 0. 
 \end{equation}
 This immediately implies \eqref{eq:TEnegative} 
 and hence completes the proof of the lemma.
\end{proof}

 \begin{remark} We note here that identity \eqref{eq:keyid} in Lemma \ref{lem:mu1}, which connects the non-degeneracy of the operator $L$ to the non-vanishing of the quantity $T_E$ has been established in other contexts. For example, in terms of traveling waves in the generalized Korteweg-de Vries equation, see \cite{jhn1}, and for its appearance in the study of periodic waves in the nonlinear Schr\"odinger equation see  \cite{GaHa2}. \end{remark}
Because \eqref{eq:q} can in fact be written as the $\nu=1$ form of Lam\'e's equation (see the Appendix for details) 
the structure of the Floquet spectrum $\Sigma({\bf Q})$ (for both librational and rotational waves) is even simpler as there is exactly one (open) gap.
\begin{figure}[h]
\begin{center}
\includegraphics{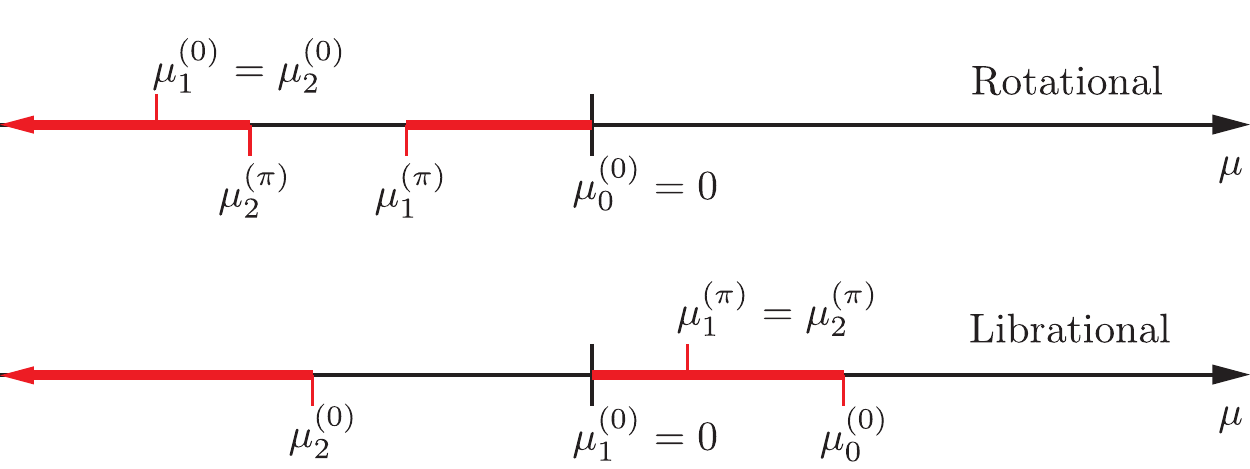}
\end{center}
\caption{A qualitative illustration of the values $\mu$ for which there exists a bounded solution to equation \eqref{eq:q}. The fact that there is a single gap corresponds to the fact that \eqref{eq:q} is reducible to Lam\'{e}'s equation. The placement of the band gap in the librational wave case is due to the fact that in this case $d\Delta_q/d\mu <0$ at $\mu=0$. 
The Floquet spectrum of \eqref{eq:q} in the $\lambda$-plane is obtained from two copies of the $\mu$-spectrum by the relation $\lambda=\pm|\gamma|^{-1}\mu^{1/2}$.
}
\label{fig:qspec}
\end{figure}

Another consequence of the preceding analysis is that it gives us a qualitative picture of the trace $\Delta_q$ for real values of $\mu$. The following result applies to all four types of periodic traveling waves of the sine-Gordon equation: 
\begin{corollary} For each periodic traveling wave $f$, there exists a point $\mu_* <0$ for which $|\Delta_q|>2$ at $\mu=\mu_*$, and therefore there exist two nonzero imaginary points $\lambda=\pm i\beta_*$, $\beta_*>0$, $\mu_*=(i\gamma\beta_*)^2$, that are \emph{not} in the spectrum of \eqref{eq:q}.
\label{cor:mustar}
\end{corollary}
\begin{remark} Actually, we have shown something stronger, namely the existence of an open interval on the negative real line where $|\Delta_q| >2$, and hence the existence of two intervals in $i\R$  not in $\Sigma \eqref{eq:q} $. This difference is not necessary to the proof of Theorem \ref{th:main} however, and we have formulated Corollary \ref{cor:mustar} in its current form as a matter of taste. 
\end{remark}

For $\mu$ inside the first spectral gap we have that $\left| \Delta_q \right|  > 2$.  Figure \ref{fig:delta} verifies that $\mu=0$ is the largest periodic eigenvalue in the rotational case, while it is the second largest in the librational case.
\begin{figure}[h]
\subfigure[rotational (subluminal)]{\includegraphics[scale=.63]{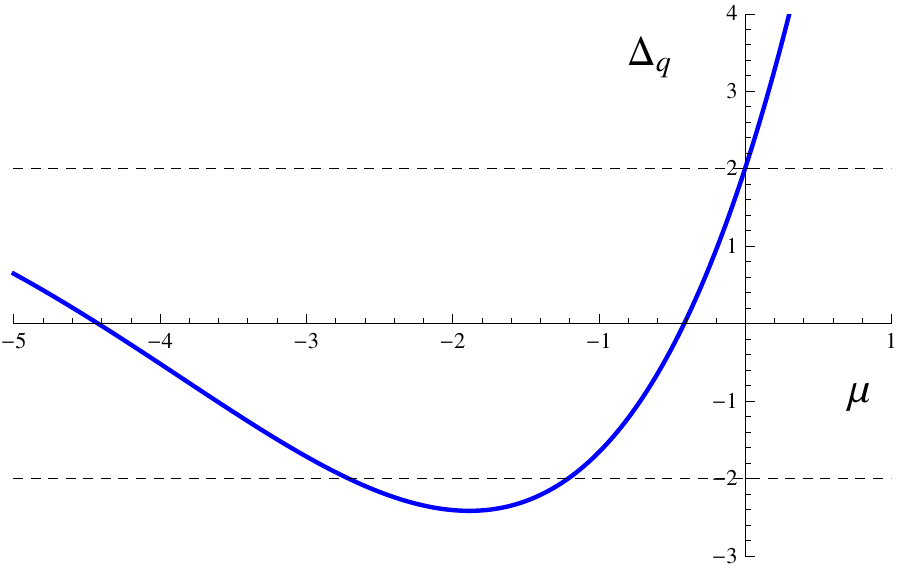}}
\subfigure[librational (superluminal)]{\includegraphics[scale=.7]{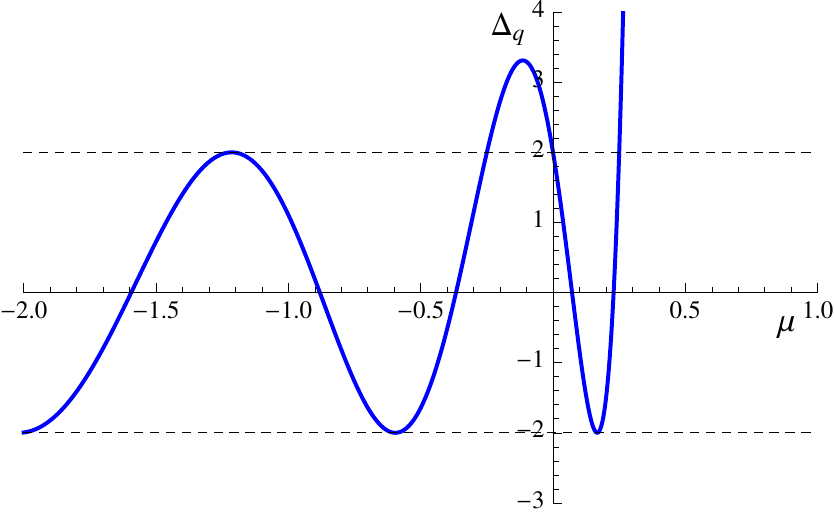}}
\caption{The graph of $\Delta_q$ for real values of $\mu$. As in Lemma \ref{lem:mu1} and equation \eqref{eq:deltaprimezero}, it can be seen that $\mu_1^{(0)} = 0$ and that $d\Delta_q/d\mu < 0$ at $\mu=0$ in the librational wave case.  Left panel:  $c=1/\sqrt{3}$, $E=4.92$.  Right panel:  $c=\sqrt{3}$, $E=1$.}\label{fig:delta}
\end{figure}


To finish off this section, we introduce another pair of functions which will prove useful in later stability calculations. 
Let $\rho_\pm$ (respectively $\eta_\pm$) denote the Floquet multipliers of \eqref{eq:p} (respectively of \eqref{eq:q}).
Define the maps $G_{p,q}:\C \to \R$ by
 \begin{equation}
 \begin{split}
 G_p(\lambda) & := \log|\rho_+(\lambda)|\log|\rho_-(\lambda)| \\
 G_q(\lambda) & := \log|\eta_+(\lambda)|\log|\eta_-(\lambda)|.
 \end{split}
 \label{eq:GpGqdef}
 \end{equation}
 \begin{lemma}[Important properties of $G_p$ and $G_q$]
 The functions $G_p:\mathbb{C}\to\mathbb{R}$ and $G_q:\mathbb{C}\to\mathbb{R}$ are continuous,  
 and $\lambda\in\sigmaP$ (respectively $\lambda\in\sigmaQ$) if and only if
 $G_p(\lambda)=0$ (respectively $G_q(\lambda)=0$).  Also, 
 \begin{equation}
 G_q(\lambda)\le 0,\quad\forall\lambda\in \C
 \label{eq:GqInequality}
 \end{equation}
 and 
 \begin{equation}
 G_p(\lambda)=(\Re(c\gamma\lambda T))^2 + G_q(\lambda),\quad \forall \lambda\in\C.
 \label{eq:GpGqIdentity}
 \end{equation}
 \label{lem:GpGq}
 \end{lemma}
 \begin{proof}
 As roots of the characteristic polynomial of a $2\times 2$ matrix with entries analytic in $\lambda$,
 the Floquet multipliers $\eta_\pm(\lambda)$ and $\rho_\pm(\lambda)$ are analytic  with the exception of square-root type branch points at the isolated zeros of the associated (entire) discriminant, and a system of branch cuts connecting them.  Moreover, the multipliers are well-defined precisely at each such branch point.  Upon crossing such a branch cut, the two multipliers are simply permuted, and hence any continuous symmetric function of the two multipliers can be continuously extended to the system of branch cuts.  
 It only remains to point out that the logarithms are continuous as the multipliers are necessarily nonzero due to \eqref{eq:determinants}.
The statement that the functions $G_p$ and $G_q$ detect the Floquet spectra of \eqref{eq:p} and \eqref{eq:q} respectively follows directly from Proposition~\ref{prop:FloquetSpectrumP} and Definition~\ref{def:tempeigenQ}.

The inequality \eqref{eq:GqInequality} follows from \eqref{eq:determinants}, which implies that the Floquet multipliers of \eqref{eq:q} are reciprocals of each other.  The identity \eqref{eq:GpGqIdentity} also follows from  the reciprocal nature of the Floquet multipliers of \eqref{eq:q} together with  Lemma~\ref{lem:floquetpandq}.
 \end{proof}

We conclude this section by deducing the sign of $G_p(\lambda)$ at certain special points in the complex plane.
\begin{lemma}\label{lem:neggp}For each periodic traveling wave there exists $\beta_*>0$ such that $G_p(\pm i\beta_*)<0$.  
\end{lemma}
\begin{proof} Let $\beta_*>0$ be as in the statement of Corollary~\ref{cor:mustar}.
Since $\lambda=i\beta_*$ is imaginary, it follows from \eqref{eq:GpGqIdentity}
that $G_p(i\beta_*)=G_q(i\beta_*)$.  But according to Corollary~\ref{cor:mustar}, $i\beta_*\not\in\sigmaQ$ and hence from Lemma~\ref{lem:GpGq} we have $G_q(i\beta_*)\neq 0$.  But from 
\eqref{eq:GqInequality} it follows that in fact $G_q(i\beta_*)<0$ and hence $G_p(i\beta_*)=G_q(i\beta_*)<0$ as desired.
\end{proof}
\begin{lemma}\label{lem:posgp} For each librational traveling wave there exists $\alpha_*\in\sigmaQ$ real and positive, and if in addition $c\neq 0$, then
$G_p(\alpha_*) > 0$ for all such $\alpha_*$. 
\end{lemma}
\begin{proof} The fact that there exist real and strictly positive points $\alpha_*$ in the Floquet spectrum
of \eqref{eq:q} for librational waves is a consequence of $\mu_0^{(0)}>0$ and the relation $\mu=(\gamma\lambda)^2$.  To see that $G_p(\alpha_*)>0$, apply  \eqref{eq:GpGqIdentity} with the observation that $G_q(\alpha_*) = 0$ and using $c\neq 0$. 
\end{proof}
%
\begin{lemma}\label{lem:largegp}
Let $\Re(\lambda)$ be sufficiently large in magnitude.  Then $\sgn(G_p(\lambda))=\sgn(\gamma)$, i.e.\ $G_p(\lambda)>0$ for superluminal waves and $G_p(\lambda)<0$ for subluminal waves.
\end{lemma}
\begin{proof} 
The coefficient $\cos(f(z))$ is evidently negligible compared with $\lambda^2$ in equation \eqref{eq:p},
and it can be shown that the Floquet multipliers of \eqref{eq:p} satisfy
\begin{equation}
\log(\rho_\pm(\lambda))=\gamma\lambda T(c\pm 1)+ O(1),\quad\lambda\to\infty.
\label{eq:logrhoasymp}
\end{equation}
Therefore, directly from the definition \eqref{eq:GpGqdef} of $G_p(\lambda)$ we see that
\begin{equation}
\begin{split}
G_p(\lambda)&=\left[\gamma\Re(\lambda) T(c+1) + O(1)\right]\left[\gamma\Re(\lambda)T(c-1)+O(1)\right] \\ &= \gamma\Re(\lambda)^2T^2 + O(\Re(\lambda)),\quad\lambda\to\infty,
\end{split}
\end{equation}
where we have used the definition \eqref{eq:defgamma} of $\gamma$.  The leading term dominates for $\Re(\lambda)$ sufficiently large, and therefore the proof is complete.
\end{proof}

\begin{remark} The function $G_{p}$ is useful not only in studying spectral stability of periodic traveling waves, but also in numerically finding the curves of spectrum. Indeed Figure~\ref{fig:spec} was made by finding the zero locus of the function $G_p(\lambda)$ with the appropriate parameters. In this way, $G_p$ is a \emph{non-analytic} but continuous analogue of an Evans function  capturing in one fell swoop the temporal eigenvalues for  \emph{all} real values of the phase $\theta$ of the (unimodular) Floquet multiplier $\rho$.  This should be contrasted with the so-called
periodic Evans function $F(\lambda;\theta)$ \cite{Grd1}.  For each $\theta\in\R$, $F(\lambda;\theta)$ is an analytic function of $\lambda$ whose (isolated) zeros are those temporal eigenvalues for which there exists a unimodular Floquet multiplier $\rho$ with phase $\theta$.
\end{remark}

\section{Stability and Instability of Periodic Traveling Waves}
\label{sec:stability_instability}
We are now ready to determine the stability and instability of the four types of periodic traveling waves. The proof of Theorem~\ref{th:main} follows from Lemma~\ref{lem:sub_rot}, Lemma~\ref{lem:librationalunstable}, and Lemma~\ref{lem:superrotunstable} that we formulate and prove in this section.

\subsection{Stability}
\begin{lemma}[Spectral stability of subluminal rotational waves]\label{lem:sub_rot} Let $f(z)$ be a subluminal rotational wave.   Then $\sigmaP$ is purely imaginary.
\end{lemma}

\begin{proof} Recall the Hill's operator $L$ defined by \eqref{eq:hillop} in \S\ref{sec:hillspec},
and the corresponding Hill's equation eigenvalue problem \eqref{eq:HillEV} parametrized
by $\theta\in\mathbb{R}$.  
The eigenvalue problem \eqref{eq:HillEV} is selfadjoint for all real $\theta$ and since the resolvent is compact (Green's function $g_\theta(z,\xi)=g_\theta(\xi,z)^*$ is a continuous and hence Hilbert-Schmidt kernel on $[0,T]^2$), it follows from the spectral theorem that for each $\theta\in\mathbb{R}$ the  eigenfunctions associated with points $\mu\in\Sigma_\theta({\bf Q})$ form an orthonormal basis of $L^2(0,T)$.  Moreover the corresponding generalized Fourier expansion of a smooth function $w(z)$ satisfying the condition $w(z+T)=e^{i\theta}w(z)$ is uniformly convergent.  From these facts it follows that
\begin{equation}
\langle w,Lw\rangle\le \|w\|^2\max\Sigma_\theta({\bf Q})\quad\forall w\in C^{(2)}(\mathbb{R}),\quad w(z+T)=e^{i\theta}w(z),
\label{eq:generalspectralbound}
\end{equation}
where
\begin{equation}
\langle w,Lw\rangle:=\int_0^T w(z)^*Lw(z)\,dz\quad\text{and}\quad\|w\|^2:=\int_0^T |w(z)|^2\,dz\ge 0.
\end{equation}
It follows from \eqref{eq:generalspectralbound} that for rotational $f$,
\begin{equation}
\langle w,Lw\rangle\le 0\quad\forall w\in C^{(2)}(\mathbb{R}),\quad w(z+T)=e^{i\theta}w(z),
\label{eq:Lposdef}
\end{equation}
because $\Sigma_\theta({\bf Q})\subset\mathbb{R}_-$.
That is, $L$ is negative semidefinite in this case (it is strictly negative definite for $\theta\neq 0\pmod{2\pi}$).  On the other hand, for librational $f$, $L$ is indefinite.

We re-write the equation \eqref{eq:p} in terms of Hill's operator $L$ as 
\begin{equation}
(c^2-1)Lp(z)-2c\lambda\frac{dp}{dz}(z) +\lambda^2 p(z)=0,
\label{eq:oureqn}
\end{equation}
\noindent and thus bounded on $\R$ solutions will satisfy 
\begin{equation}
p(z+T)=e^{i\theta}p(z), \;\;\text{for $\theta\in \R$.}
\label{eq:ourbc}
\end{equation}

For a given $\theta\in\R,$ we define $\sigma_\theta({\bf P}) \subset \mathbb{C}$ to be the set of complex $\lambda$ for which there exists a nontrivial solution to \eqref{eq:oureqn} satisfying the boundary condition \eqref{eq:ourbc}.  The Floquet spectrum of \eqref{eq:p} is the union over $\theta$ of these sets:
\begin{equation}
\sigma({\bf P})=\bigcup_{-\pi<\theta\le\pi}\sigma_\theta({\bf P}).
\end{equation}
  
Suppose $\lambda\in\sigma({\bf P})$.  Then there exists $\theta\in\mathbb{R},$ such that in fact $\lambda \in \sigma_\theta({\bf P})$.  Let $p\in C^{(2)}(\mathbb{R})$ denote the corresponding eigenfunction satisfying \eqref{eq:oureqn} and \eqref{eq:ourbc}.  Then,
multiplying the differential equation through by $p(z)^*$ and integrating over the fundamental period $[0,T]$ gives
\begin{equation}
(c^2-1)\langle p,Lp\rangle -2im\lambda+\|p\|^2\lambda^2=0,
\label{eq:lambdaquadratic}
\end{equation}
where
\begin{equation}
m:=-ic\int_0^T p(z)^*\frac{dp}{dz}(z)\,dz\in\mathbb{R}.
\end{equation}
That $m$ is real follows by integration by parts
using the $\theta$-periodicity condition satisfied by $p$.
The relation \eqref{eq:lambdaquadratic} can be viewed as a quadratic equation for $\lambda$. Expressing $\lambda$ in terms of $\langle p,Lp\rangle$, $m$, and $\|p\|^2$, we have:
\begin{equation}
\lambda=\frac{1}{\|p\|^2}\left[im \pm\sqrt{-m^2-(c^2-1)\|p\|^2\langle p,Lp\rangle}\right],
\label{eq:quadraticsoln}
\end{equation}
and given the reality of $m$ these values are clearly purely imaginary as long as $c^2<1$ (the subluminal case) and $f$ is a rotational wave (implying the negative semidefiniteness condition $\langle p,Lp\rangle\le 0$ according to \eqref{eq:Lposdef}). 
\end{proof}

\subsection{Instability}
Instability of the sub- and superluminal librational waves and of the superluminal rotational waves   is, in each case, a consequence of the continuity of the function $G_p$ defined by \eqref{eq:GpGqdef} and the Intermediate Value Theorem.  
\begin{lemma}[Spectral instability of librational waves]
Let $f(z)$ be a librational wave.  Then there exists a temporal eigenvalue $\lambda_*\in\sigmaP$ with $\Re(\lambda_*)>0$.
\label{lem:librationalunstable}
\end{lemma}
\begin{proof}
If $c=0$, then it is obvious that \eqref{eq:defineq} is the identity transformation and hence equations \eqref{eq:p} and \eqref{eq:q} coincide implying $\sigmaP=\sigmaQ$.   Lemma~\ref{lem:posgp} then implies the existence of a real positive temporal eigenvalue $\lambda_*=\alpha_*>0$.

If $c\neq 0$ (but $c^2\neq 1$), let $\beta_*>0$ be as in the statement of Lemma~\ref{lem:neggp} and let $\alpha_*>0$ be as in the statement of Lemma~\ref{lem:posgp}.  Choose a continuous
mapping (parametrized curve) $\lambda:[0,1]\to\C$ for which $\lambda(0)=\lambda_0:=i\beta_*$ and $\lambda(1)=\lambda_1:=\alpha_*$,
and for which $\Re(\lambda(t))>0$ for $0<t\le 1$.  (To be concrete, one might select the straight line $\lambda(t)=\lambda_0 (1-t) + \lambda_1t$.)  Then $G_p(\lambda(t))$ is a continuous function from $[0,1]$ to $\R$, and $G_p(\lambda(0))<0$ from Lemma~\ref{lem:neggp} while $G_p(\lambda(1))>0$ from Lemma~\ref{lem:posgp}.  It follows from the Intermediate Value Theorem that there exists $t_*\in (0,1)$ such that $G_p(\lambda(t_*))=0$, and therefore
$\lambda_*:=\lambda(t_*)$ has a positive real part and $\lambda_*\in\sigmaP$.
\end{proof}
\noindent
The main idea of the proof is illustrated in the left-hand panel of Figure~\ref{fig:IVT}.
\begin{figure}[h]
\begin{center}
\includegraphics{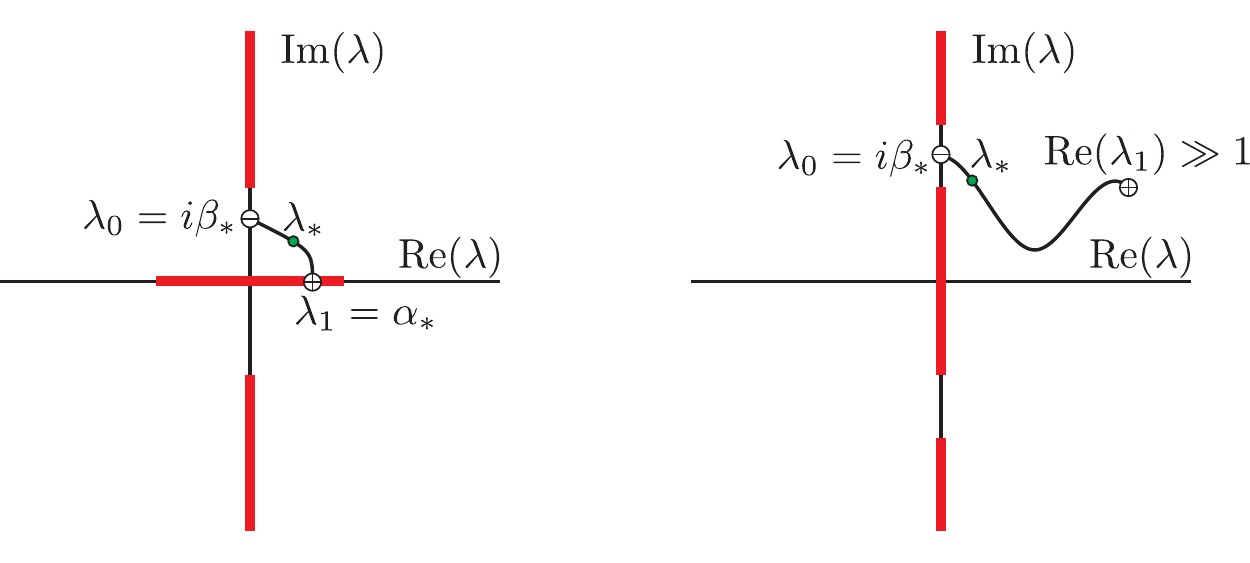}
\end{center}
\caption{Left:  the spectrum $\sigmaQ$ for a librational wave shown in bold (red), and a curve $\lambda=\lambda(t)$ connecting $\lambda_0=i\beta_*$ at which $G_p<0$ with $\lambda_1=\alpha_*>0$ at which $G_p>0$, resulting in a root $\lambda_*$ of $G_p$ (point of $\sigmaP$) in the right half-plane.  Right:  the spectrum $\sigmaQ$ for a rotational wave shown in bold (red), and
a curve $\lambda=\lambda(t)$ connecting $\lambda_0=i\beta_*$ at which $G_p<0$ with a point
$\lambda_1$ with $\Re(\lambda_1)\gg 1$ at which $G_p>0$ if the wave is superluminal, resulting in a root $\lambda_*$ of $G_p$ in the right half-plane.}
\label{fig:IVT}
\end{figure}

\begin{lemma}[Spectral instability of superluminal rotational waves]
Let $f(z)$ be a rotational wave and suppose that $c^2>1$.  Then there exists a temporal eigenvalue $\lambda_*\in\sigmaP$ with $\Re(\lambda_*)>0$.
\label{lem:superrotunstable}
\end{lemma}
\begin{proof}
Again take $\beta_*>0$ as in the statement of Lemma~\ref{lem:neggp} so that with $\lambda_0:=i\beta_*$ we have $G_p(\lambda_0)<0$.
Pick $\lambda_1$ with $\Re(\lambda_1)>0$ so large that according to Lemma~\ref{lem:largegp}
we have $G_p(\lambda_1)>0$ (because $c^2>1$). 
Choosing any continuous curve $\lambda:[0,1]\to\C$ with $\lambda(0)=\lambda_0$, $\lambda(1)=\lambda_1$,
and $\Re(\lambda(t))>0$ for $0<t\le 1$ and applying the Intermediate Value Theorem to $G_p(\lambda(t))$,  the rest of the proof is exactly as in that of Lemma~\ref{lem:librationalunstable}.
\end{proof}
\noindent The proof is illustrated in the right-hand panel of Figure~\ref{fig:IVT}.
Combining Lemma~\ref{lem:sub_rot}, Lemma~\ref{lem:librationalunstable}, and Lemma~\ref{lem:superrotunstable} completes the proof of Theorem~\ref{th:main}.

\section{Some Results on the Structure of the Floquet Spectrum}
\label{sec:structure}
In this final section, we go beyond the basic (in)stability result of Theorem~\ref{th:main} to deduce some qualitative features of the Floquet spectrum $\sigmaP$ that can be easily obtained by virtually the same methodology already in place.  The first result in this direction is the following.
\begin{proposition}
Let $f(z)$ be a subluminal librational wave.  Then there exists a positive real temporal eigenvalue $\lambda_*\in\sigmaP$, and one of the corresponding Floquet multipliers satisfies $\rho(\lambda_*)=1$ (i.e.\ $\lambda_*$ is a periodic Floquet eigenvalue of \eqref{eq:p}).
\end{proposition}
\begin{proof}
If $c=0$, then equations \eqref{eq:p} and \eqref{eq:q} coincide and hence the positive
real periodic eigenvalue $\lambda_*:=(\mu_0^{(0)}/\gamma^2)^{1/2}>0$ of \eqref{eq:q}
is also a point of $\sigmaP$.  Therefore both Floquet mulipliers of the equivalent problems \eqref{eq:p} and \eqref{eq:q} are equal to $1$ for $\lambda=\lambda_*>0$.

If $c\neq 0$, let $\lambda_0:=\alpha_*(\mu_0^{(0)}/\gamma^2)^{1/2}>0$ be as in the statement of Lemma~\ref{lem:posgp}, so that as $\lambda_0$ is the most positive real point in $\sigmaQ$ we have $G_p(\lambda_0)>0$.  Let $\lambda_1>0$ be so large that by Lemma~\ref{lem:largegp}
we have $G_p(\lambda_1)<0$ (because the wave is subluminal and hence $\gamma<0$).  
Then applying the Intermediate Value Theorem to the continuous function $G_p(\lambda_0(1-t)+\lambda_1t)$ we find a root $t_*\in (0,1)$, and since $\lambda_*:=\lambda_0(1-t_*)+\lambda_1t_*>0$, it is a strictly positive real temporal eigenvalue of \eqref{eq:p}.  It follows from Lemma~\ref{lem:realp} that $\lambda_*$ must be either a periodic or antiperiodic temporal eigenvalue of \eqref{eq:p}.  To see that $\lambda_*$ is in fact a periodic eigenvalue, note that
since $\lambda_0$ is the most positive real value for which the Floquet multipliers $\eta$ of Hill's equation \eqref{eq:q} coincide (with value $\eta=1$), it follows from Lemma~\ref{lem:floquetpandq} that the Floquet multipliers $\rho$ of equation \eqref{eq:p} are coincident at $\lambda_0$ with value
$\rho=e^{c\gamma\lambda_0T}>0$ and are distinct for all $\lambda>\lambda_0$, either forming a distinct real pair or a complex-conjugate pair.  From \eqref{eq:logrhoasymp} in the proof of Lemma~\ref{lem:largegp} it is clear that for sufficiently large real $\lambda$ the multipliers are
real and distinct, and it follows that they are in fact real and distinct for all $\lambda>\lambda_0$.  From
\eqref{eq:determinants} it is obvious that the product of the real multipliers $\rho$ is positive for
all $\lambda\in\R$ and therefore they have the same definite sign (positive, by consideration of the degenerate value for $\lambda=\lambda_0$) for all $\lambda\ge\lambda_0$.  Since one of the multipliers $\rho$ satisfies $\rho^2=1$ for $\lambda=\lambda_*>\lambda_0$ it follows that in fact $\rho=1>0$.
\end{proof}

The next result concerns the nature of the spectrum $\sigmaP$ near the origin in the complex $\lambda$-plane; temporal eigenvalues can only accumulate at $\lambda=0$ (see Figure \ref{fig:lib_spec_origin}).
\begin{proposition}
Let $f$ be a librational wave of any speed $c\neq \pm 1$, and let $\sigmaP^\circ$ denote the corresponding Floquet spectrum of \eqref{eq:p} with points in the union of the real and imaginary $\lambda$-axes omitted.  There is an open neighborhood $U$ of the origin such that $\lambda=0$ is the only limit point of $\sigmaP^\circ\cap U$ lying in the union of the real and imaginary axes.
\label{prop:limit}
\end{proposition}
\begin{proof}
By Proposition~\ref{lem:symmetryp} it suffices to consider the part of $\sigmaP^\circ$ in the first
quadrant of the complex $\lambda$-plane.  Let $\lambda_0=i\beta_*$ with $\beta_*>0$ being as in the statement of Lemma~\ref{lem:neggp}
and let $\lambda_1=\alpha_*$ with $\alpha_*>0$ being as in the statement of Lemma~\ref{lem:posgp}.  Let $U$ be an open disk centered at the origin with radius $\tfrac{1}{2}\min\{|\lambda_0|,|\lambda_1|\}$.  

We first prove that there exist limit points of $\sigmaP^\circ$ on the union of the real and imaginary $\lambda$-axes.  Consider the one-parameter family of curves connecting $\lambda_0$ with $\lambda_1$ with parameter $p>0$:  $\lambda=\lambda_p(t):=\lambda_0(1-t)^p+\lambda_1t^p$ for $0\le t\le 1$.  Exactly as in the proof of Lemma~\ref{lem:librationalunstable}, the Intermediate Value Theorem provides for each $p>0$ a point $\lambda=\lambda_{*,p}\in\sigmaP$ that is an interior point of the curve $\lambda=\lambda_p(t)$ (see the left-hand panel of Figure~\ref{fig:lib_spec_origin}).  As a bounded subset of $\C$, the sequence $\{\lambda_{*,p}\}_{p=1}^\infty$ has limit points, and as the curve $\lambda=\lambda_p(t)$ with $0\le t\le 1$ approaches the union of the real segment $[0,\alpha_*]$ and the imaginary segment $[0,i\beta_*]$ as $p\to\infty$, the limit points of the sequence necessarily lie on this union of segments.

Now let $\lambda$ denote \emph{any} limit point of $\sigmaP^\circ$ lying on the union of the segments $[0,\alpha_*]$ and $[0,i\beta_*]$.  Because $\sigmaP$ is a closed set, it follows that $\lambda\in\sigmaP$ as well.  If $\Im(\lambda)>0$ and $\Re(\lambda)=0$,
then as $\Im(\lambda)<\beta_*$ we have $\lambda\not\in\sigmaQ$ which by Corollary~\ref{cor:pandqimag} contradicts the fact that $\lambda\in\sigmaP$.  On the other hand, if $\Re(\lambda)>0$ and $\Im(\lambda)=0$, then as $\Re(\lambda)<\alpha_*$ we have  $\lambda\in\sigmaQ$ which by Lemma~\ref{lem:pandqimag} contradicts again the fact that $\lambda\in\sigmaP$.  It therefore follows that $\lambda=0$.
\end{proof}

\begin{figure}[h] 
\begin{center}
\includegraphics{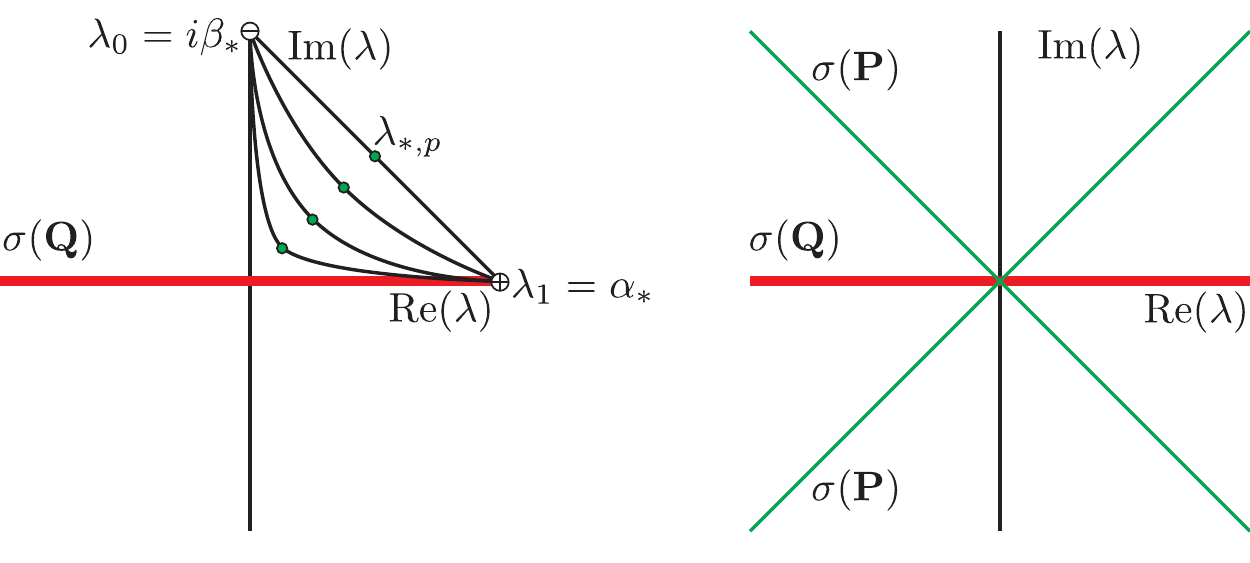}
\end{center}
\caption{Left:  a magnification of the neighborhood of the origin from the left-hand panel of Figure~\ref{fig:IVT} showing the family $\lambda=\lambda_p(t)$ of curves from the proof of Proposition~\ref{prop:limit} and the corresponding points $\lambda_{*,p}\in\sigmaP$ converging to the origin.  Right:  the local structure of the spectrum near the origin for librational waves (both sub- and superluminal) consists of a union of two crossing analytic arcs.}
\label{fig:lib_spec_origin}
\end{figure}

It is an easy consequence of the fact that the elements of the monodromy matrix $\M_p(\lambda)$ are entire analytic functions of $\lambda$ that for each $R>0$, $\sigmaP\cap \{|\lambda|<R\}$ is 
a finite union of analytic arcs (\emph{arcs of spectrum}).  Together with Proposition~\ref{prop:limit} this implies that for $R>0$ sufficiently small, $\sigmaP\cap\{|\lambda|<R\}$ is a finite union of arcs of spectrum meeting at the origin and having the full Hamiltonian symmetry implied by Proposition~\ref{lem:symmetryp}.   To deduce further details of the local structure of $\sigmaP$ near $\lambda=0$ (for example the number of arcs of spectrum meeting at the origin and their asymptotic form in the limit of small $|\lambda|$) requires methods beyond those presented in this paper.  However, in a forthcoming paper \cite{JMMP2} we carry out a complete local analysis
by computing the Taylor expansion of $\M_p(\lambda)$ about $\lambda=0$, and this method shows that for librational waves, $\sigmaP$ near the origin consists of exactly two curves crossing transversely with opposite slopes as illustrated qualitatively in the right-hand panel of Figure~\ref{fig:lib_spec_origin}.  The nature of the Floquet spectrum near the origin is of crucial importance in relating spectral stability/instability to Whitham's so-called \emph{modulational stability criterion} \cite{Wh} (see also \cite{JZ1}).

\section*{Acknowledgements}
RM would like to thank Jared Bronski, Chris Cosgrove, Georg Gottwald, Yuri Latushkin, and Davide Masoero for many productive conversations regarding the spectrum of the periodic sine-Gordon operator, and would like to acknowledge the support of the Australian Research Council grant ARC- DP110102775. PDM thanks Richard Koll\'ar for useful conversations and gratefully acknowledges the support of the National Science Foundation under grants DMS-0807653 and DMS-1206131.  RGP warmly thanks Tim Minzoni and Panayiotis Panayotaros for enlightening conversations on periodic traveling waves and their stability.

\section*{Appendix:  Writing \eqref{eq:q} as Lam\'e's Equation}
Lam\'e's equation \cite[\S29.2(i)]{NIST:DLMF} is
\begin{equation}
\frac{d^2w}{d\zeta^2}+(h-\nu(\nu+1)k^2\mathrm{sn}^2(\zeta,k))w=0,
\label{eq:Lame}
\end{equation}
where $\nu\ge -\tfrac{1}{2}$ and $k\in (0,1)$ is an elliptic modulus, and where $\mathrm{sn}(\zeta,k)$ denotes the Jacobi elliptic function with argument $z$ and modulus $k$, which satisfies the first-order nonlinear equation \cite[\S22.13(ii)]{NIST:DLMF}
\begin{equation}
\left(\frac{du}{d\zeta}\right)^2=(1-u^2)(1-k^2u^2).
\label{eq:sn-eqn}
\end{equation}
The fact that elliptic functions are involved essentially follows from \eqref{eq:pendint} satisfied by the traveling wave profile $f(z)$.  The calculations are slightly different for each of the four types of periodic traveling waves.
\subsection*{Subluminal ($\gamma<0$) rotational ($E<0$) waves}
Consider the relation
\begin{equation}
\cos(f(z))=-1+2u^2.
\label{eq:subrot-relation}
\end{equation}
Differentiating \eqref{eq:subrot-relation}, squaring, using the Pythagorean identity $\sin^2(f)+\cos^2(f)=1$, and substituting again from \eqref{eq:subrot-relation} yields the equation
\begin{equation}
4\left(\frac{du}{dz}\right)^2=(1-u^2)\left(\frac{df}{dz}\right)^2.
\label{eq:star-one}
\end{equation}
Next, substitution of \eqref{eq:subrot-relation} into \eqref{eq:pendint} gives
\begin{equation}
\left(\frac{df}{dz}\right)^2=-2\gamma(2-E)(1-k^2u^2),\quad k:=\sqrt{\frac{2}{2-E}}\in (0,1).
\end{equation}
Eliminating $f'(z)^2$ therefore yields \eqref{eq:sn-eqn} with $\zeta=[\tfrac{1}{2}(-\gamma)(2-E)]^{1/2}(z-z_0)$ for arbitrary constant $z_0\in\mathbb{R}$, and hence
\begin{equation}
u=\mathrm{sn}\left(\left[\frac{1}{2}(-\gamma)(2-E)\right]^{1/2}(z-z_0),\sqrt{\frac{2}{2-E}}\right).
\end{equation}
With the use of \eqref{eq:subrot-relation} it is then clear that \eqref{eq:q} takes the form of Lam\'e's equation \eqref{eq:Lame} with
\begin{equation}
k:=\sqrt{\frac{2}{2-E}},\quad h:=k^2\left[1-\frac{2\mu}{(-\gamma)}\right],\quad\zeta:=\left[\frac{1}{2}(-\gamma)(2-E)\right]^{1/2}(z-z_0),\quad \nu=1.
\end{equation}
\subsection*{Superluminal ($\gamma>0$) rotational ($E>2$) waves}
Consider instead the relation
\begin{equation}
\cos(f(z))=1-2u^2,
\label{eq:suprot-relation}
\end{equation}
which by the same steps as in the subluminal rotational case implies \eqref{eq:star-one}.  Using \eqref{eq:suprot-relation} in \eqref{eq:pendint} gives
\begin{equation}
\left(\frac{df}{dz}\right)^2=2\gamma E(1-k^2u^2),\quad k:=\sqrt{\frac{2}{E}}\in (0,1).
\end{equation}
Using this to eliminate $f'(z)^2$ from \eqref{eq:star-one} then yields \eqref{eq:sn-eqn} with $\zeta=[\tfrac{1}{2}\gamma E]^{1/2}(z-z_0)$ for constant $z_0\in\mathbb{R}$.  Therefore
\begin{equation}
u=\mathrm{sn}\left(\left[\frac{1}{2}\gamma E\right]^{1/2}(z-z_0),\sqrt{\frac{2}{E}}\right),
\end{equation}
and using \eqref{eq:suprot-relation} shows that in this case \eqref{eq:q} is also of the form of Lam\'e's equation
\eqref{eq:Lame} with
\begin{equation}
k:=\sqrt{\frac{2}{E}},\quad h:=k^2\left[1-\frac{\mu}{\gamma}\right],\quad\zeta:=\left[\frac{1}{2}\gamma E\right]^{1/2}(z-z_0),\quad \nu=1.
\end{equation}
\subsection*{Subluminal ($\gamma<0$) librational ($0<E<2$) waves}
The correct substitution in this case is
\begin{equation}
\cos(f(z))=-1+(2-E)u^2
\label{eq:sublib-relation}
\end{equation}
which in combination with \eqref{eq:pendint} leads to \eqref{eq:sn-eqn} with $k=[\tfrac{1}{2}(2-E)]^{1/2}\in (0,1)$ and $\zeta=[-\gamma]^{1/2}(z-z_0)$ and hence
\begin{equation}
u=\mathrm{sn}\left(\left[-\gamma\right]^{1/2}(z-z_0),\sqrt{\frac{2-E}{2}}\right).
\end{equation}
Using \eqref{eq:sublib-relation} then shows that \eqref{eq:q} takes the form of Lam\'e's equation
\eqref{eq:Lame} with
\begin{equation}
k:=\sqrt{\frac{2-E}{2}},\quad h:= 1-\frac{\mu}{(-\gamma)},\quad\zeta:=[-\gamma]^{1/2}(z-z_0),\quad \nu=1.
\end{equation}
\subsection*{Superluminal ($\gamma>0$) librational ($0<E<2$) waves}
Finally, in this case we use
\begin{equation}
\cos(f(z))=1-Eu^2,
\label{eq:suplib-relation}
\end{equation}
which together with \eqref{eq:pendint} gives \eqref{eq:sn-eqn} with $k=[\tfrac{1}{2}E]^{1/2}\in (0,1)$
and $\zeta=\gamma^{1/2}(z-z_0)$.  Therefore
\begin{equation}
u=\mathrm{sn}\left(\gamma^{1/2}(z-z_0),\sqrt{\frac{E}{2}}\right).
\end{equation}
Combining this result with \eqref{eq:suplib-relation} shows that \eqref{eq:q} is again Lam\'e's equation \eqref{eq:Lame} with
\begin{equation}
k:=\sqrt{\frac{E}{2}},\quad h:=1-\frac{\mu}{\gamma},\quad\zeta:=\gamma^{1/2}(z-z_0),\quad \nu=1.
\end{equation}
\subsection*{The Floquet spectrum of Lam\'e's equation for $\nu=1$}
In all four cases we have Lam\'e's equation \eqref{eq:Lame} with the constant parameter $\nu=1$.
The Floquet spectrum of Lam\'e's equation is well-understood, and the relevant information
can be found in \cite[\S29.3(i)--(ii), \S29.9]{NIST:DLMF}.  In passing from the notation of \cite{NIST:DLMF} to our notation, we need to recall the relationship between $\mu$ and $h$ in each case (note in particular that $dh/d\mu<0$ in all cases) and also keep in mind that for both types of librational waves the fundamental period $T$ we use to compute the monodromy actually corresponds to
\emph{twice} the fundamental period of the coefficient in Lam\'e's equation.  It then follows that
since $\nu=1$, the only open gap in $\Sigma({\bf Q})$ (not counting the ``trivial gap'' $\mu>\mu_0^{(0)}$) is the interval $\mu_2^{(\pi)}<\mu<\mu_1^{(\pi)}$
for both types of rotational waves and the interval $\mu_2^{(0)}<\mu<\mu_1^{(0)}$ for both types of librational waves.  These exact results are illustrated in Figure~\ref{fig:qspec}.

\bibliography{JMMP}
\bibliographystyle{amsalpha}


\end{document}